\documentclass[a4paper,10pt,reqno]{amsart}
\usepackage{amscd,amssymb,graphicx,color,epigraph}

\usepackage[shortalphabetic,initials]{amsrefs}

\usepackage{enumitem}
\setlist[itemize]{leftmargin=18pt}
\setlist[enumerate]{leftmargin=18pt}
\usepackage{comment}

\usepackage{amssymb}
\usepackage{amsthm}
\usepackage{bm}
\usepackage{latexsym}
\usepackage{amsmath}
\usepackage{hyperref}
\usepackage{eufrak}
\usepackage{mathrsfs}
\usepackage{amscd}
\usepackage[all,cmtip]{xy}
\usepackage{xcolor}
\usepackage{color}
\usepackage{amscd}
\usepackage{listings}
\usepackage{framed}
\usepackage[center]{caption}
\theoremstyle{plain}

 \numberwithin{equation}{section}

\newtheorem{theorem}{Theorem}[section]
\newtheorem{proposition}[theorem]{Proposition}

\newtheorem{Question}[theorem]{Question}
\newtheorem{lemma}[theorem]{Lemma}

\newtheorem{corollary}[theorem]{Corollary}

\theoremstyle{definition}

\newcommand{\appsection}[1]{\let\oldthesection\thesection
\renewcommand{\thesection}{Appendix \oldthesection}
\section{#1}\let\thesection\oldthesection}

\newtheorem{definition}[theorem]{Definition}

\newtheorem{remark}[theorem]{Remark}

\DeclareMathOperator{\exc}{Exc}
\DeclareMathOperator{\sing}{Sing}
\DeclareMathOperator{\Def}{Def}

\def\D{{\mathbb{D}}}

\def\Z{{\mathbb{Z}}}

\def\Q{{\mathbb{Q}}}
\def\C{{\mathbb{C}}}
\def\P{{\mathbb{P}}}

\def\W{{\mathscr{W}}}
\def\O{{\mathcal{O}}}
\def\T{{\mathbb{T}}}

\def\QG{{\text{QG}}}

\makeatletter
\providecommand{\leftsquigarrow}{%
  \mathrel{\mathpalette\reflect@squig\relax}%
}
\newcommand{\reflect@squig}[2]{%
  \reflectbox{$\m@th#1\rightsquigarrow$}%
}
\makeatother

\title{Exotic surfaces}

\author[Javier Reyes]{Javier Reyes}
\email{jereyes@umd.edu}
\address{Department of Mathematics, University of Maryland, 4176 Campus Drive - William E. Kirwan Hall, College Park, USA.}

\author[Giancarlo Urz\'ua]{Giancarlo Urz\'ua}
\email{urzua@mat.uc.cl}
\address{Facultad de Matem\'aticas, Pontificia Universidad Cat\'olica de Chile, Campus San Joaqu\'in, Avenida Vicu\~na Mackenna 4860, Santiago, Chile.}
\date{\today}

\begin{document}

\begin{abstract}
We construct the first exotic $\C \P^2 \# 4 \overline{\C \P^2}$ by means of rational blowdowns. Similarly, we construct the first exotic $3\C \P^2 \allowbreak \# \allowbreak b^- \overline{\C \P^2}$ for $b^-=9,8,7$. All of them are minimal and symplectic, as they are produced from projective surfaces $W$ with Wahl singularities, and $K_W$ big and nef. In more generality, we elaborate on the problem of finding exotic $(2\chi(\O_W)-1) \C \P^2 \# (10\chi(\O_W)-K^2_W-1) \overline{\C \P^2}$ from these Koll\'ar--Shepherd-Barron--Alexeev surfaces $W$, obtaining explicit geometric obstructions.
\end{abstract}

\maketitle

\section{Introduction} \label{s1}

Let $b^+,b^-$ be positive integers. An exotic $b^+ \C \P^2 \# b^- \overline{\C \P^2}$ is a differentiable 4-manifold that is homeomorphic but not diffeomorphic to $b^+ \C \P^2 \# b^- \overline{\C \P^2}$. There are many examples and constructions of exotic $b^+ \C \P^2 \# b^- \overline{\C \P^2}$, with a particular emphasis on $b^+=1$, i.e. exotic blow-ups of $\P_{\C}^2$ (see e.g. \cite{K89}, \cite{P05}, \cite{SS05}, \cite{PSS05}, \cite{FS06}, \cite{P07}, \cite{A08}, \cite{AP08}, \cite{ABP08}, \cite{AP10}, \cite{FS11}, \cite{BK17}, \cite{BKS22}; with a complex structure see e.g. \cite{LP07}, \cite{PPS09a}, \cite{PPS09b}, \cite{PPS13}, \cite{RU21}). Several of these constructions involve the rational blowdown technique of Fintushel and Stern \cite{FS97}, generalized by Park \cite{P97} which, in the context of algebraic geometry, corresponds to $\Q$-Gorenstein smoothings of Wahl singularities \cites{W81,LW86,SSW08}. The rational blowdown surgery is realized along the exceptional divisor of the minimal resolution of a Wahl singularity. This divisor is called a Wahl chain, and it is a linear configuration of $\P_{\C}^1$'s. Due to Symington \cites{Sy98,Sy01}, the rational blowdown of a disjoint collection of Wahl chains on a complex surface admits a symplectic structure. 

Exotic $b^+ \C \P^2 \# b^- \overline{\C \P^2}$ with small $b^+$ and $b^-$ (i.e. small topological Euler characteristic) tend to be difficult to find, and it seems that there might not be a satisfactory explanation for this. Constructions of exotic $\C \P^2 \# b^- \overline{\C \P^2}$ do exist for every $b^-\geq 2$, but it is an open question the existence of an exotic $\C \P^2 \# \overline{\C \P^2}$. Similarly, there are minimal symplectic examples of $3 \C \P^2 \# b^- \overline{\C \P^2}$ for any $b^-\geq 4$, but it seems to be unknown for $b^-=3,2,1$. In these two families, one can verify that the number of examples decreases rapidly as $b^-$ gets smaller. If we only consider constructions for $b^+ = 1$ where the rational blowdown is applied, then there are very few exotic $\C \P^2 \# 5 \overline{\C \P^2}$ \cite{PPS09b,BKS22}, and no examples of exotic $\C \P^2 \# b^- \overline{\C \P^2}$ for $b^- \leq 4$. For instance, in \cite[Question 2]{BKS22} it is asked: \textit{Is there an exotic $\C \P^2 \# b^- \overline{\C \P^2}$ with $b^- < 5$ that can be obtained from a standard
rational surface via rational blowdowns? If so, what is the smallest such $m$?} In this paper we show that $b^-=4$ is possible.

\begin{theorem}
There exists a minimal, symplectic, and exotic $\C \P^2 \# 4 \overline{\C \P^2}$ which is constructed via rational blowdown from a rational surface.
\end{theorem}

Moreover, and in more generality, in this paper we elaborate on all the details involved in the construction of such exotic 4-manifolds. Hopefully, this can give us hints on why small exotic 4-manifolds are hard to find, and how to deal with the corresponding difficulties. We study this problem in connection with the Koll\'ar--Shepherd-Barron--Alexeev (KSBA) surfaces, which compactify the Gieseker moduli space of surfaces of general type \cite{KSB88,K23}. To explain this connection concretely, let $X$ be a nonsingular projective surface with a collection of disjoint Wahl chains. There exists a contraction $X \to W$ of these Wahl chains, which produces a normal projective surface $W$ with only Wahl singularities \cite{Art62}. In Theorem \ref{min} and Corollary \ref{exotic} we prove a slightly more general version of the following.

\begin{theorem}
Assume that $K_W$ is ample (i.e. $W$ is a KSBA surface). Then the rational blowdown $Y$ of the Wahl chains in $X$ is a minimal symplectic 4-manifold. In addition, if $Y$ is simply-connected and has an odd intersection form, then $Y$ is an exotic $b^+ \C \P^2 \# b^- \overline{\C \P^2}$ where $b^+=2\chi(\O_W)-1$, and $b^-=10\chi(\O_W)-K^2_W-1>0$.
\end{theorem}

The possibilities for $X$ are limited. It must have irregularity equal to zero, and it is birational to either $\P_{\C}^2$, a K3 surface, an Enriques surface, a proper elliptic surface, or some surface of general type (see Proposition \ref{typeofZ}). Moreover, due to the positivity of $K_W$, if $\sigma \colon X \to Z$ is some composition of blow-downs of $(-1)$-curves, then the image under $\sigma$ of the collection of Wahl chains is a non-empty configuration of rational curves. Thus, the big goal is to identify conditions over such configurations in $Z$ that can guarantee the existence of $W$.

We take the point of view of geography of configurations of curves, and the invariants we study are their logarithmic Chern numbers. As the singularities of a configuration of curves can become too complicated, we restrict our configurations to the case of nodal rational curves with simple crossings singularities (Definition \ref{nsc}). The study of these configurations, their log Chern numbers, and the connection with the existence of $W$ is in Section \ref{s3}, whose main result is Theorem \ref{bmy}. It can be summarized as follows.

\begin{theorem}
Let $W$ be a normal projective surface with $\ell$ Wahl singularities, and $K_W$ big and nef. Let $\pi \colon X \to W$ be its minimal resolution. Assume that there is a birational morphism $\sigma \colon X \to Z$ such that the image under $\sigma$ of the Wahl chains is a nodal simple crossings configuration of $r$ rational curves. Then, $$K_W^2 \leq 12(1+p_g(Z)) -\frac{1}{3} K_Z^2 - \frac{2}{3} \ell - \frac{1}{3}r + \frac{1}{3}\sum_{k\geq 3} (k-2) t_k^0.$$ Here $t_k^0$ is the number of exceptional curves over k-points with $k> 2$ which are not part of the Wahl chains.
\end{theorem}

For example, if $Z=\P_{\C}^2$, then $K_W^2 \leq 9 - \frac{2}{3} \ell - \frac{1}{3}r + \frac{1}{3}\sum_{k\geq 3} (k-2) t_k^0$. 

Another way of obtaining constraints is via the $\Q$-Gorenstein deformation space of $W$. To explain this briefly, let $T_W$ be the tangent sheaf $\mathcal{H} \text{om}_{\O_W}(\Omega_W^1,\O_W)$, where $\Omega_W^1$ is the sheaf of differentials on $W$. Assume that $W$ is a normal projective surface with $\ell$ Wahl singularities, and $K_W$ big. Then one can show that (see Theorem \ref{obstr}) $$h^2(T_W)= \ell + h^1(T_W) + 2K_W^2-10\chi(\O_W).$$ This dimension measures obstructions to deform $W$. If $h^2(T_W)=0$, then $W$ admits $\Q$-Gorenstein smoothings, and so the rational blowdown $Y$ has a complex structure. Therefore if we have $10\chi(\O_W) \leq 2K_W^2$, then $h^2(T_W)>0$ and the configuration of rational curves in $Z$ produces obstructions to deform $W$. We elaborate on $\Q$-Gorenstein deformations of $W$ in Section \ref{s4}. For one of the constructions of an exotic $\C \P^2 \# 4 \overline{\C \P^2}$ we compute that $H^1(W,T_W)=0$ (so $W$ is equisingularly rigid), and $H^2(T_W)=\C^2$ (see Proposition \ref{obtrex}). We do not know if this surface $W$, and any of the examples in this paper, admits a complex structure.

Similarly, but starting with a particular configuration of rational curves in a special K3 surface, we prove in Section \ref{s6} the following.

\begin{theorem}
There exist minimal, symplectic, and exotic $3\C \P^2 \# b^- \overline{\C \P^2}$ for $b^-=9,8,7$ which are constructed via rational blowdown from a composition of blow-ups starting with a K3 surface.
\end{theorem}

We point out that we have constructed many different surfaces $W$ that can be used to produce small exotic 4-manifolds, and in this paper we present some of them. For example, there are hundreds of surfaces $W$ that yield exotic $\C \P^2 \# 5 \overline{\C \P^2}$, for $\C \P^2 \# 6 \overline{\C \P^2}$ we can produce thousands, and so on. We may try to describe all of them in some future work. In this paper, we exhibit 17 distinct surfaces $W$ to prove our main theorems. The problem of finding the appropriate configuration of rational curves involves a computer program, written from scratch in C++ by the first-named author as part of his master's thesis at UC Chile. This program was fed with the data of configuration of curves to systematically search for special Wahl chains among all possibilities. It was run on 80 cores at the MiDaS cluster at UC. An interesting combinatorial problem here is: \textit{How to classify the configurations of rational curves that produce KSBA surfaces $W$?}

\subsubsection*{Acknowledgments} We thank R. \.{I}nan\c{c} Baykur, Jonny Evans, Paul Hacking, Pedro Montero, and Jenia Tevelev for valuable discussions. The first-named author was funded by the ANID scholarship 22201484. The second-named author was supported by the FONDECYT regular grant 1190066. This work was also supported by ANID - Millennium Science Initiative Program - grant NCN$17\textunderscore \,059$ ``Millenium Nucleus Center for the Discovery of Structures in Complex Data" \href{https://midas.mat.uc.cl/index.html}{MiDas}.

\tableofcontents

\section{KSBA surfaces and exotic $4$-manifolds} \label{s2}

Let $0<q<m$ be coprime integers. A \textit{cyclic quotient singularity} is a singularity locally isomorphic to the germ at $(0,0)$ of the quotient of $\C^2$ by the action of $(x,y) \mapsto (\mu x, \mu^q y)$, where $\mu$ is an $m$-th primitive root of $1$. The notation for this singularity is $\frac{1}{m}(1,q)$.

Any normal two-dimensional singularity $(P \in W')$ admits a \textit{minimal resolution}, this is, a birational morphism $\pi \colon X \to W'$ such that $X$ is nonsingular, the pre-image of $P$ is a collection of exceptional curves $\exc(\pi)$, none of them being a $(-1)$-curve \footnote{In a nonsingular surface a $(-p)$-curve is a $\P_{\C}^1$ whose self-intersection is equal to $-p$.}, and $\pi|_{X \setminus \exc(\pi)}$ is an isomorphism. If $(P \in W')$ is a cyclic quotient singularity of type $\frac{1}{m}(1,q)$, then $\exc(\pi)$ is a chain of $\P^1_{\C}$'s $E_1,\ldots,E_s$ such that $E_i^2=-e_i \leq -2$, where the $e_i$ are obtained through the associated Hirzebruch-Jung continued fraction $\frac{m}{q}=[e_{1},\ldots,e_{s}]$. We define its length as $s$. We have the numerical relation $K_X \equiv \pi^*(K_{W'}) + \sum_{i=1}^s d_{i} E_{i}$, where $-1<d_{i} \leq 0$ are the discrepancies of the exceptional curves.

\begin{definition}
A \textit{T-singularity} is a cyclic quotient singularity of type $\frac{1}{dn^2}(1,dna-1)$ where $d \geq 1$, and $0<a<n$ with gcd$(n,a)=1$. A \textit{Wahl singularity} is a T-singularity with $d=1$. Its exceptional divisor in its minimal resolution will be called \textit{Wahl chain}.
\label{wsing}
\end{definition}

\begin{remark}
Any T-singularity or Du Val singularity $(P \in W')$ can be partially resolved $\phi \colon W \to W'$ so that $\phi^*(K_{W'}) \equiv K_W$. We call them \textit{M-resolutions} \cite{BC94}. For a Du Val singularity, its M-resolution is just the minimal resolution. For a T-singularity $\frac{1}{dn^2}(1,dna-1)$, its M-resolution has as exceptional divisor a chain of $d-1$ $\P_{\C}^1$'s passing through $d$ Wahl singularities $\frac{1}{n^2}(1,na-1)$. 
\end{remark}

T-singularities and Du Val singularities are key to describing all deformations of an arbitrary quotient singularity by means of its P-resolutions. This was approached in the foundational work of Koll\'ar and Shepherd-Barron \cite{KSB88}, where a compactification for the moduli space of surfaces of general type is introduced, nowadays known as the KSBA compactification. Earlier than that work, Wahl introduces T-singularities in \cite{W81} as examples admitting a smoothing with constant $K^2$, and in \cite[Proposition 5.9]{LW86} Looijenga and Wahl show that T-singularities (and Du Val type A singularities) are the only cyclic quotient singularities admitting a smoothing with that property. In \cite{KSB88} it is proved that T-singularities and Du Val singularities are precisely the quotient singularities that admit a $\Q$-Gorenstein smoothing, this is, a deformation over a disk with nonsingular general fibers so that the underlying 3-fold has a $\Q$-Cartier canonical class. 

One way to construct new complex surfaces is by considering $\Q$-Gorenstein smoothings of a surface $W$ with only Wahl singularities. This approach has been successful in finding new surfaces of general type with particular properties (see e.g. \cite{LP07}, \cite{PPS09a}, \cite{PPS09b}). To guarantee the existence of a smoothing, it is typically shown that there a no local-to-global obstructions to deforming $W$. This is a strong condition for surfaces in general, where it is not easy to prove the existence of complex smoothings in other ways. 

Instead, Fintushel and Stern \cite{FS97} (see also Park \cite{P97} for a complete generality) introduced the \textit{rational blowdown} technique. This is a surgical procedure on a closed 4-manifold that contains disjoint Wahl chains, obtaining a new smooth closed 4-manifold by removing a small neighborhood of the Wahl chains and attaching rational homology balls instead. These are the Milnor fibers of the $\Q$-Gorenstein smoothings of the corresponding Wahl singularities. (See \cite{SSW08} for more on the relation between these two viewpoints.) In \cite{Sy98,Sy01}, Symington shows that one can give a symplectic structure to the rational blowdown, provided that we start with a symplectic 4-manifold and the 2-spheres in the Wahl chains are symplectic.

A nonsingular complex projective surface $X$ admits a symplectic form inherited from some embedding in projective space. Thus, if $X$ contains disjoint Wahl chains, then we can apply Symington's theorem to construct a symplectic rational blowdown $Y$ along the Wahl chains in $X$.

We recall that a divisor $D$ on an irreducible projective surface $Z$ is nef if $D \cdot \Gamma \geq 0$ for all curves $\Gamma \subset Z$. In addition $D$ is big if and only if $D^2>0$ \cite[Theorem 2.2.16]{L04}. 

\begin{theorem}
Let $W'$ be a normal projective surface with only T-singularities and Du Val singularities such that $K_{W'}$ is ample. Let $\phi \colon W \to W'$ be the M-resolution of $W'$, and let $\pi \colon X \to W$ be the minimal resolution of $W$. Then the rational blowdown $Y$ of $\exc(\pi)$ is a (smoothly) minimal symplectic 4-manifold.
\label{min}
\end{theorem}

\begin{proof}
Let $E_i$ be the curves in $\exc(\pi)$, and let $F_i$ (resp. $G_i$) be the proper transforms in $X$ of the exceptional curves of $\phi$ over T-singularities (resp. Du Val singularities). First we claim there is $N>0$, and integers $n_i,m_i,s_i \geq 0$ such that $$H:= \pi^*(NK_{W}) - \sum_i n_i E_i - \sum_i m_i F_i - \sum_i s_i G_i$$ is ample. We know that there are integers $n_i,m_i,s_i \geq 0$ such that $$\big(\sum_i n_i E_i + \sum_i m_i F_i + \sum_i s_i G_i \big) \cdot \Gamma <0,$$ for all $\Gamma=E_j,F_j,G_j$ as the intersection matrices for exceptional curves over singularities are negative definite. 

We know that $K_W \equiv \phi^*(K_{W'})$ is nef and big. Then by \cite[Proposition 2.2.6]{L04} (Kodaira's lemma), there is $M >0$ (sufficiently divisible) such that $$H^0\big(X,\pi^*(M K_{W}) - \sum_i n_i E_i - \sum_i m_i F_i - \sum_i s_i G_i \big) \neq 0,$$ and so we have that $H_0=\pi^*(M K_{W}) - \sum_i n_i E_i - \sum_i m_i F_i - \sum_i s_i G_i$ is linearly equivalent to some fixed effective divisor. Thus there could be a finite number of curves (in that fixed divisor) which may have a nonpositive intersection with $H_0$. As $H_0 \cdot \Gamma >0$ for all $\Gamma=E_j,F_j,G_j$ and they are the only curves that intersect zero with $\pi^*(M K_{W})$ (for the rest of curves it has to be positive as $K_{W'}$ is ample), we can find $N>M$ so that $H=H_0 + \pi^*((N-M) K_{W})$ intersects any curve in $X$ positively, and $H^2 >0$ (as $K_{W'}$ is ample). Hence by the Nakai-Moishezon criterion \cite[Theorem 1.2.23]{L04}, we see that $H$ is ample. 

Let us assume we took a higher multiple to make $H$ very ample. In this way, we have an embedding of $X$ into projective space via $H$. Let $X^0$ be $X$ minus the Wahl chains in $\exc(\pi)$, and let $W^0$ be $W$ minus the Wahl singularities. We have that $H|_{X^0}$ is $\pi^*(NK_{W})|_{X^0} - \sum_i m_i F_i - \sum_i s_i G_i$. We have the symplectic form $\omega_X$ induced by the complex structure in the class of $H$. As $Y$ is the rational blowdown of $X$ at the Wahl chains in $\exc(\pi)$, we have that $H_2(W^0,\Z) \subset H_2(Y,\Z)$, and we also have $H_2(Y,\Z) \subset H_2(W,\Z)$ as the Milnor fibers a rational homology balls. Therefore, the Symington's symplectic form $\omega_Y$ on $Y$ is in the class represented by $NK_Y - \sum_i m_i F_i - \sum_i s_i G_i$.

By \cite[Proposition 2.1]{Li06}, if $Y$ is not minimal (in a smooth sense), then there is a symplectic $(-1)$-sphere $E$ in $Y$. But then $$0<\int_{E} \omega_Y|_{E}=E \cdot (NK_Y - \sum_i m_i F_i - \sum_i s_i G_i) <0,$$ as none of the $F_i$, $G_i$ are $(-1)$-spheres, which is a contradiction. Therefore $Y$ is minimal.


\end{proof}

\begin{remark}
We note that if $W'=W$, i.e. only Wahl singularities on $W'$, then $\omega_Y$ is in the class of $N K_Y$. We can also get a symplectic form $\omega_Y$ in the class of $N K_Y$ starting with the $W'$ in Theorem \ref{min}, by means of constructing $Y$ via cutting the T-singularities and Du Val singularities from $W'$, and then glueing the Milnor fibres corresponding to the $\Q$-Gorenstein smoothings with their exact symplectic forms. \label{jonny} 
\end{remark}

Throughout this paper we will only consider those $Y$ that are simply connected.

We recall that $\chi_{\text{top}}(Y)=\chi_{\text{top}}(W)$ and $K_Y^2=K_W^2$ for the rational blowdown in Theorem \ref{min}. Let $(b^+,b^-)$ be the signature of the unimodular quadratic form on $H^2(Y,\Z)$. Then we have the second Betti number $b_2(Y)= b^++b^-=\chi_{\text{top}}(W)-2$, and for its signature we have $b^+-b^-=\frac{1}{3}(K^2_W-2\chi_{\text{top}}(W))$. Therefore $$ b^+=2\chi(\O_W)-1 \ \ \ \ \ b^-=10\chi(\O_W)-K^2_W-1,$$ as the Noether's formula $12 \chi(\O_W)=K_W^2+\chi_{\text{top}}(W)$ holds for $W$. In particular, the intersection form cannot be negative definite. It cannot be positive definite either as that would imply $K_W^2=5\chi_{\text{top}}(W)-6$, and by the orbifold Bogomolov-Miyaoka-Yau inequality (see e.g. \cite{L03}) we would have $5\chi_{\text{top}}(W)-6 \leq 3 \chi_{\text{top}}(W)-t$, where $t>0$, and so $b^+<1$, a contradiction. Therefore the quadratic form must be indefinite. By Freedman's theorem, the oriented homeomorphism type of $Y$ is completely determined by its intersection form, and as it must be indefinite, it is determined by $b^+$, $b^-$, and its parity (see e.g. \cite[Chapter IX]{BHPV04}).

In this paper, we will only consider the case when the intersection form is odd, so $Y$ is homeomorphic to $b^+ \C \P^2 \# b^- \overline{\C \P^2}$ with $b^->0$. A consequence of \cite[Proposition 2.1]{Li06} is the following.

\begin{corollary}
Let $Y$ be the minimal symplectic 4-manifold constructed in Theorem \ref{min}. If $Y$ is simply-connected and it has an odd intersection form, then $Y$ is an exotic $b^+ \C \P^2 \# b^- \overline{\C \P^2}$  with $b^->0$.  \label{exotic}   
\end{corollary}

\section{KSBA surfaces and geography of rational configurations} \label{s3}

Let $W$ be a normal projective surface with $\ell$ Wahl singularities, and $K_W$ big and nef. We will consider the following diagram of morphisms $$ \xymatrix{  & X  \ar[ld]_{\sigma} \ar[rd]^{\pi} &  \\ Z &  & W}$$ where $\pi$ is the minimal resolution of $W$ (thus $\exc(\pi)$ consists of $\ell$ Wahl chains), and $\sigma$ is some composition of blow-ups at nonsingular points starting with a nonsingular projective surface $Z$. 

\begin{proposition}
The surface $Z$ satisfies one of the following:
\begin{itemize}
    \item It is a rational surface.
    \item It is a blow-up of either a K3 surface, or an Enriques surface.
    \item It has Kodaira dimension equal to $1$ and $b_1(Z)=0$.
    \item It is of general type, $b_1(Z)=0$, and $K_Z^2 < K_W^2$.
\end{itemize}
Moreover, the configuration of rational curves $\sigma(\exc(\pi))$ is non-empty.
\label{typeofZ}
\end{proposition}

\begin{proof}
This is the content of \cite[Proposition 2.3]{RU19} when $W$ has one Wahl singularity and $K_W$ is ample. But the same proof works for many Wahl singularities, and $K_W$ big and nef except in the general type case. The inequality is \cite[Lemma 2.2]{R17}. The claim $b_1(Z)=0$ can be shown via the Albanese variety of $Z$. For a proof of the last statement see \cite[Proposition 2.3]{RU21}.
\end{proof}

Evans and Smith \cite{ES20} found a bound for the length of the Hirzebruch-Jung continued fraction of each of the singularities in $W$ when $p_g(W)>0$ (see also \cite[Theorem 1.1]{RU19} when $W$ has only one singularity). They prove that if this length is $s$, then $s \leq 4 K_W^2 +7$. It turns out that one can replace the hypothesis ``$p_g>0$" with ``rational blowup is not a rational surface" in \cite[Theorem 1.1]{ES20}. This is done in \cite[Theorem A.3]{CZ20} by Chen and Zhang. In addition, it is possible to improve that bound by modifying \cite[Propositions 8.3 and 8.4]{ES20} following what is done in \cite[Lemma 2.8]{RU19}, so the bound ends up being $$s \leq 4 K_W^2 +1.$$ This bound is optimal by \cite{RU19}, and it is not true for $W$ rational. Two boundedness problems on these surfaces $W$ remain open:
\begin{enumerate}
    \item Find bounds depending on $K_W^2$ which involve all singularities at once, and
    \item Find bounds depending on $K_W^2$ when $W$ is a rational surface.
\end{enumerate}
Problem (1) involves hard combinatorics, as one has to deal with ``bad" configurations of $(-1)$-curves and Wahl chains in $X$. That analysis would allow us to know and classify those bad configurations. Problem (2) was partially worked out in \cite[Theorem 1.3]{RU19}, but we do not have yet a real bound depending only on $K_W^2$.
\smallskip

We are going to focus on the configuration of rational curves $\sigma(\exc(\pi))$. The objective is to find a connection between the problem of the existence of a $W$ with the required properties, and log Chern numbers of $\sigma(\exc(\pi))$. In this way, by applying the log BMY inequality, we will produce constraints for the existence of $W$.    

We recall the log Chern numbers of a configuration of curves $\{C_1,\ldots,C_r\}$ in a surface $Z$; Cf. \cite{Urz10}. Let $\phi \colon \tilde Z \to Z$ be a log resolution of $D=\sum_{i=1}^r C_i$, in the sense that $\phi$ is a birational morphism such that the (reduced) pull-back $\tilde D$ of $D$ has only simple normal crossings, and any exceptional $(-1)$-curve in $\tilde D$ intersects at least two other curves in $\tilde D$. Then the log Chern numbers of $(Z,D)$ are defined via the Chern classes of $\Omega_{\tilde Z}^1(\log \tilde D)$ as $\bar{c}_1^2(Z,D):= c_1(\Omega_{\tilde Z}^1(\log \tilde D)^*)^2$ and $\bar{c}_2(Z,D):=c_2(\Omega_{\tilde Z}^1(\log \tilde D)^*)$. We have $\bar{c}_1^2(Z,D)=(K_{\tilde{Z}}+\tilde D)^2$ and $\bar{c}_2(Z,D)=\chi_{\text{top}}(\tilde Z \setminus \tilde D)$. It can be proved that they do not depend on the log resolution. (See \cites{EFU22,Urz22} for a systematic study of these invariants when $Z=\P^2$ and $D$ is a configuration of lines.) We will only consider the following type of configurations.

\begin{definition}
Let $Z$ be a nonsingular projective surface. A \textit{nodal simple crossings configuration} is a collection of projective curves $C_1,C_2,\ldots,C_r$ such that:  
\begin{itemize}
    \item The curve $C_i$ is nodal and rational with $\nu_i$ nodes, and
    \item At a common point of $C_i$ and $C_j$ there are different tangent directions.
\end{itemize}
\label{nsc}
\end{definition}
    
Hence, the divisor $D=\sum_{i=1}^r C_i$ has only simple crossings singularities, which are locally of the form $(x-a_1 y)(x-a_2 y)\cdots(x-a_k y)=0$ at $(0,0) \in \C^2$ with $a_i \neq a_j$. Such a point will be called \textit{k-point}. The number of k-points in $D$ will be denoted by $t_k$. We denote the number of nodes as $$\nu := \sum_i \nu_i.$$ The numbers $\nu$ and $t_2$ are not necessarily comparable.

\begin{proposition}
We have $\bar{c}_1^2(Z,D)= K_Z^2 - \sum_{i=1}^r C_i^2 + \sum_{k\geq 2} (3k-4) t_k -4r +2\nu,$ and $\bar{c}_2(Z,D)= \chi_{\text{top}}(Z) + \sum_{k\geq 2} (k-1)t_k - 2 r.$ \label{peo}
\end{proposition}

\begin{proof}
Let $Z' \to Z$ be the blow-up at all nodes from the $C_i$. These nodes could be isolated, or they could be at k-points with $k>2$. Let $D'$ be the reduced pull-back of $D$ in $Z'$. By \cite{Urz10}, There are formulas for the log Chern numbers of $(Z',D')$, and they translate to $$\bar{c}_1^2(Z,D)= K_Z^2 - \sum_{i=1}^r C_i^2 + \sum_{k\geq 2} (3k-4) t_k + 4 \sum_{i=1}^r (p_a(C_i)-1) -2 \nu,$$ and $$\bar{c}_2(Z,D)= \chi_{\text{top}}(Z) + \sum_{k\geq 2} (k-1)t_k + 2 \sum_{i=1}^r (p_a(C_i)-1) -2 \nu.$$ But the curves $C_i$ are rational, and so $p_a(C_i)=\nu_i$ for all $i$.
\end{proof}

\begin{definition}
Let $\{{C'}_1,\ldots,{C'}_{r'}\}$ be any configuration of curves on a nonsingular surface $Z'$, and let $D'=\sum_{i=1}^{r'} {C'}_i$. We define $$ P(Z',D'):= \sum_{i=1}^{r'} \left({C'}_i^2 + 5 - \sum_{P \in \sing(D')} m_P({C'}_i)^2\right),$$ where $m_P(C)$ is the multiplicity of $P \in C$, and $$ K(Z',D'):= K_{Z'}^2 + 2 r'- |\sing(D')|-P(Z',D').$$  Note that if $D'$ is a collection of $\ell$ disjoint Wahl chains, then $P(Z',D')=\ell$ and $K(Z',D')=K_{Z'}^2+r'$ by the particular properties of Wahl chains \cite[Proposition 3.11]{KSB88}.
\label{pk}
\end{definition}

\begin{proposition}
Let $D'=C'_1+\ldots+C'_{r'}$ be any configuration of curves on a nonsingular surface $Z'$. If $E$ is a nonsingular curve not in $D'$, intersecting it at $n$ nonsingular points of $D'$ and $m$ singular points of $D'$, then $P(Z',D'+E)=P(Z',D')+E^2+5-2n-m$, and $K(Z',D'+E)=K(Z',D')-E^2-3+n+m$.
\label{add} 
\end{proposition}

\begin{proposition}
Let $D'=C'_1+\ldots+C'_{r'}$ be any configuration of curves on a nonsingular surface $Z'$. Let $Z'' \to Z'$ be the blow-up at a point $P$ in $\sing(D')$, and let $D''$ be the strict transform of $D'$. Then $P(Z'',D'')=P(Z',D')$ and $K(Z'',D'')=K(Z',D')$. In particular, if the singularity of $D'$ at $P$ is a node and $E$ is the exceptional $(-1)$-curve, then $P(Z'',D''+E)=P(Z',D')$ and $K(Z'',D''+E)=K(Z',D')$. 
\label{blow}
\end{proposition}

\begin{proof}
Both propositions are direct from the formulas in Definition \ref{pk}.
\end{proof}

From now on, let $D=\sum_{i=1}^r C_i$ be a nodal simple crossings configuration on $Z$.

\begin{proposition}
We have $$P(Z,D)= \sum_{i=1}^r C_i^2 +5r - \sum_{k\geq 2} k t_k - 2 \nu,$$ and $$K(Z,D)= K_Z^2 -3r - \sum_{i=1}^r C_i^2 + \sum_{k\geq 2} (k-1)t_k +2\nu.$$ 
\end{proposition}

Let $\sigma' \colon Z' \to Z$ be the blow-up at all the k-points of $D$ with $k>2$. Let us choose $t_k^0$ of the $t_k$ exceptional curves for each $k>2$. Let $D'$ be the (reduced) pull-back of $D$ minus the $\sum_{k>2} t_k^0$ chosen exceptional curves. Then, by Propositions \ref{add} and \ref{blow}, we can compute $$P(Z',D')=P(Z,D)- \sum_{k\geq 3} (2k-4)(t_k-t_k^0),$$ and $K(Z',D')= K(Z,D) + \sum_{k\geq 3} (k-2) (t_k -t_k^0)$. 

We note that further blow-ups at the notes of $D'$ and additions of the corresponding exceptional curves will not change $P$ and $K$.

\begin{theorem}
Let $W$ be a normal projective surface with $\ell$ Wahl singularities, and $K_W$ big and nef. Let $\pi \colon X \to W$ be its minimal resolution. Assume that there is a birational morphism $\sigma \colon X \to Z$ such that $\sigma(\exc(\pi))$ is a nodal simple crossings configuration of rational curves $D=C_1+\ldots+C_r$. Then, $$K_W^2 \leq 12(1+p_g(Z)) -\frac{1}{3} K_Z^2 - \frac{2}{3} \ell - \frac{1}{3}r + \frac{1}{3}\sum_{k\geq 3} (k-2) t_k^0.$$ Moreover, $$\bar{c}_1^2(Z,D)=K_Z^2-\ell+r + 2 \sum_{k \geq 3} (k-2)t_k^0,$$ $$\bar{c}_2(Z,D)=12(1+p_g(Z))-\ell-K_W^2 + \sum_{k \geq 3} (k-2)t_k^0,$$ and $$2r=K_W^2-K_Z^2+\ell+\sum_{k\geq 2} (k-1)t_k - \sum_{k \geq 3} (k-2)t_k^0.$$ As we defined earlier, here $t_k^0$ is the number of exceptional curves over $k$-points with $k> 2$ which are not part of $\exc(\pi)$.
\label{bmy}
\end{theorem}

\begin{proof}
As before, let $\sigma' \colon Z' \to Z$ be the blow-up at all the k-points of $D$ with $k>2$. We have already computed $P(Z',D')$ and $K(Z',D')$ where $D'$ is the pull-back of $D$ taking out the exceptional curves of $\sigma'$ that will not be part of the Wahl chains in $X$. As we want to calculate the log Chern numbers of the pair $(Z,D)$, we will consider all exceptional divisors of the log-resolution $Z' \to Z$ of $D$. It turns out that including them all imposes more restrictions on $K_W^2$ than not considering some of them.

On the other hand, we know that $P(Z',D')=\ell$ and $K(Z',D')=K_W^2$ by Proposition \ref{blow}, and the observation at the end of Definition \ref{pk}. We can isolate the terms $\sum_k k t_k$ and $\sum_k t_k$ from the formulas for $P(Z',D')$ and $K(Z',D')$, and evaluate them on the formulas for $\bar{c}_1^2(Z,D)$ and $\bar{c}_2(Z,D)$ to obtain the last three claims in the theorem. Here we are using that $b_1(Z)=0$, and Noether's formula $12 \chi(\O_Z)=K_Z^2+\chi_{\text{top}}(Z)$ as well.

For the first claim, we apply the log Bogomolov-Miyaoka-Yau inequality as expressed by Langer in \cite[Theorem 0.1]{L03}. For that we need to have a log canonical pair $(X,D'')$ so that a multiple of $K_X+D''$ is effective. We are taking as $D''$ the (reduced) pull-back of $D$ by $\sigma$. Hence $D''=\exc(\pi)+\sum_i G_i$ where $G_i$ are the exceptional curves of $\sigma$. We also know that $K_X \equiv \pi^*(K_W) + \sum_i d_i C_i$, where $-1<d_i<0$, $C_i \in \exc(\pi)$, and by assumption $\pi^*(K_W)$ is big and nef, so it is a sum of an ample divisor plus some effective divisor. Therefore some multiple of $K_X + D''$ is indeed effective. 
\end{proof}

\begin{corollary}
With the same hypothesis of Theorem \ref{bmy}, assume in addition that the image of $\exc(\pi)$ under $\sigma$ has only nodes as singularities. Then $$ 7 K_W^2 \leq 12 + 6b_2(Z) + 5K_Z^2 - 5 \ell - t_2.$$   
\end{corollary}

\begin{corollary}
With the same hypothesis of Theorem \ref{bmy}, assume in addition that $t_k^0=0$ for all $k>2$ and $K_Z^2=0$. Then $$ 7 K_W^2 \leq 12 + 6b_2(Z) - 5 \ell - \sum_{k\geq 2} (k-1) t_k.$$ 
\end{corollary}


\section{Obstructions to complex smoothings} \label{s4}

In this section, we analyze the potential complex deformations of a normal projective surface $W$ with $\ell$ Wahl singularities, and $K_W$ big and nef based on \cite{H04, H12}. As in the introduction, let $T_W:=\mathcal{H} \text{om}_{\O_W}(\Omega_W^1,\O_W)$ where $\Omega_W^1$ is the sheaf of differentials on $W$.

A deformation (proper, flat morphism) $(W \subset \W) \to (0 \in \D)$ over a nonsingular curve germ $\D$ is $\Q$-Gorenstein if $K_{\W}$ is $\Q$-Cartier. Hence, for each singularity we have either a locally trivial deformation, or a smoothing with Milnor number zero. It is called $\Q$-Gorenstein smoothing if in addition $W_t$ (the fiber over $t$) is nonsingular for $t \neq 0$. We assume that $W_0$ is a reduced fiber isomorphic to $W$. We know that there is a versal deformation space $\Def_{QG}(W)$ which parametrizes all $\Q$-Gorenstein deformations of $W$. A main question for us is whether this space contains smoothings. For this, we want to study $\T_{QG,W}^0$, the vector space of infinitesimal automorphisms of $W$; $\T_{QG,W}^1$, the vector space of $\Q$-Gorenstein first order deformations of $W$; and $\T_{QG,W}^2$, the vector space of obstructions for $\Q$-Gorenstein deformations.

In general, it is a theorem that for proper schemes $W$ over a field $k$, Aut$(W)$ is a group scheme locally of finite type over $k$, and its tangent space at the identity is $H^0(W,T_W)$. In our case, since $k=\C$, Aut$(W)$ is reduced, so if Aut$(W)$ is finite, then $H^0(W,T_W)=0$. But Aut$(W)$ is indeed finite when $K_W$ is big by a well-known theorem of Iitaka \cite[Section 11.1 for definitions, Theorem 11.12 for the result]{Ii82}. On the other hand, it is shown in \cite[Lemma 3.8]{H04} that $\T_{QG,W}^0=H^0(W,T_W)$, as there is an isomorphism of sheaves $T_W \simeq \mathcal{T}_{QG,W}^0$. Therefore $\T_{QG,W}^0=0$.

For the analysis of the two other vector spaces, we first consider the minimal resolution of singularities $\pi \colon X \to W$. We have that $\pi_* T_X = T_W$ (see \cite[Proposition 1.2]{BW74}). By the Leray spectral sequence of low-degree terms (see \cite[Lemma (2.1.3)]{NSW08}, \cite[Section 3]{H12}), we have the exact sequence $$0 \to H^1(W,T_W) \to \T^1_{QG,W} \to H^0(W, \mathcal{T}^1_{QG,W}) \to H^2(W,T_W) \to \T^2_{QG,W} \to 0.$$ This happens because $\mathcal{T}^i_{QG,W}$ is supported on the isolated singularities of $W$ for $i>0$, and $\mathcal{T}^2_{QG,W}=0$ as the local canonical covers of $W$ are complete intersections (see \cite[p.227]{H04}). The general set-up is described in \cite[Section 3]{H12}. The vector space $H^0(W, \mathcal{T}^1_{QG,W})$ is the direct sum of the one-dimensional tangent spaces of the $\Q$-Gorenstein deformations of each Wahl singularity, i.e. $H^0(W, \mathcal{T}^1_{QG,W})=\C^{\ell}$. If $\T^2_{QG,W}=0$, then $\Def_{QG}(W)$ is nonsingular (this happens, for example, when $H^2(W,T_W)=0$). The vector space $H^1(W,T_W)$ parametrizes equisingular deformations of $W$. These dimensions are related to the invariants of $W$ as follows. 

\begin{theorem}
Let $W$ be a normal projective surface with $\ell$ Wahl singularities, and $K_W$ big. Then $$ \dim_{\C} \T^1_{QG,W} = 10 \chi(\O_W)-2 K_W^2 + \dim_{\C} \T^2_{QG,W},$$ and so $h^2(T_W)= \ell + h^1(T_W) + 2K_W^2-10\chi(\O_W)$. 
\label{obstr}
\end{theorem}

\begin{proof}
Let us consider the minimal resolution $\pi \colon X \to W$ with exceptional divisor $E= \sum_{i=1}^{\ell} E_i$, where $E_i$ is the exceptional Wahl chain over each Wahl singularity. In particular, $E$ is a simple normal crossings divisor. We have the short exact sequence $$ 0 \to  T_X(-\log E) \to T_X \to \bigoplus_{C \in E} N_{C / X} \to 0,$$ where $N_{C / X} \simeq \O_C(C)$ is the normal sheaf of $C$ in $X$. As in \cite[Lemma 1]{LP07}, we have $R^i \pi_* T_X(-\log E)=0 $ for $i=1,2$, and we know that $\pi_* T_X = T_W$. This implies $$ h^2(T_W)-h^1(T_W)=\chi(W,T_W)= \chi(X,T_X(-\log E)),$$ as $h^0(T_W)=0$. We now use the Hirzebruch-Riemann-Roch theorem for the locally free sheaf $T_X(-\log E)$. Its exponential Chern character is $$\text{ch}(T_X(-\log(E)))=2-(K_X+E)+\frac{1}{2} \big((K_X+E)^2 - 2 \chi_{\text{top}}(X \setminus E) \big).$$ From this we compute $\chi(X,T_X(-\log(E)))=2 K_W^2 - 10 \chi(\O_W) + \ell$. The result follows from the Leray exact sequence above. 
\end{proof}

If $\text{dim}_{\C} H^2(W,T_W)=0$, then $\Def_{QG}(W)$ is nonsingular and any local deformations of the singularities in $W$ may be glued to obtain a global deformation of $W$. We say in this case that there are no \textit{local-to-global} obstructions to deform $W$. For instance, if $K_W^2=5$ and $p_g(W)=q(W)=0$, then we have $h^2(W,T_W)>0$ and $\text{dim}_{\C} \T^1_{\QG,W} = \text{dim}_{\C} \T^2_{\QG,W}$.

For the examples in the following sections, we will start with an elliptic fibration, and we will consider singular fibers and multi-sections as part of the configurations of rational curves. We also want to concretely compute $H^2(T_W)$, which is the same as computing $H^2(X,T_X(-\log E))$ as we just saw.

Let $Z \to \P_{\C}^1$ be an elliptic fibration with sections, and $D$ a configuration of rational curves $D$ in $Z$. Consider a composition of blow-ups $$\sigma \colon X = Z_n \to Z_{n-1} \to \ldots \to Z_0 = Z,$$ as in the previous section, where each intermediate $Z_{i+1} \to Z_i$ is the blow-up at a single point. We successively construct a divisor $D_{i+1}$ in $Z_{i+1}$ with curves in the proper transform of $D_i$ plus some exceptional curves. We will keep track of $H^2(Z_i, T_{Z_i}(-\log D_i))$ as we add or erase curves, applying the following principles. The dimension of $H^2(Z_{i+1},T_{Z_{i+1}}(-\log D_{i+1}))$ is equal to the dimension of $H^2(Z_{i},T_{Z_{i}}(-\log D_{i}))$ when $D_i$ has simple normal crossings and $D_{i+1}$ is the reduced pullback of $D_i$. Furthermore, the dimension of $H^2(Z_{i},T_{Z_{i}}(-\log D_{i}'))$ is equal to the dimension of $H^2(Z_{i},T_{Z_{i}}(-\log D_{i}))$ if

\begin{itemize}
    \item[(-1)] $D_i'$ is obtained from $D_i$ by adding or erasing $(-1)$-curves that are transversal to the rest of $D_i$. This applies to removing the $(-1)$-curve of the reduced pullback of $D_{i-1}$, if the blow-up was made at a node. (see e.g. \cite[Proposition 4.2 and 4.3]{PSU13}.

    \item[(-2)] $D_i'$ is obtained from $D_i$ by adding any Du Val configuration of $(-2)$-curves that are disjoint from $D_i$ \cite[Theorem 4.4]{PSU13}.
\end{itemize}

\begin{lemma}
Let $f \colon Z \to \P_{\C}^1$ be an elliptic fibration with at least one section. Assume that it has $n$ fibers $F_i$ of type $I_1$ and $m$ fibers $G_j$ of type $I_k$ with $k>1$. (They may not be all the singular fibers of $f$.) Let $\sigma' \colon Z' \to Z$ be the blow-up at the nodes of the $F_i$. Let $F_i$ (resp. $G_i$) be the proper transform of $F_i$ (resp. $G_i$), and let $H_i$ be the exceptional curves of $\sigma'$. Then $$H^2\Big(Z',T_{Z'}\big(-\log(\sum_{i=1}^n (F_i+H_i) + \sum_{i=1}^m G_i \big)\Big)\geq m+n-2+p_g(Z).$$ 
\label{ineq} 
\end{lemma}

\begin{proof}
This follows \cite[Proposition 7]{LP07}. We have $K_{Z'} \sim (p_g(Z)-1)F + \sum_{i=1}^n H_i$, where $F$ is a general fiber of $Z' \to \P_{\C}^1$. Consider the inclusion $$f^*(\Omega_{\P_{\C}^1}^1 (n+m)) \otimes K_{Z'} \to \Omega_{Z'}^1\big(\log(\sum_{i=1}^n (F_i+H_i) + \sum_{i=1}^m G_i \big) \otimes K_{Z'}.$$ The inequality follows by Serre duality. 
\end{proof}

\begin{lemma}
With the same hypothesis as in the previous lemma, let us assume that $p_g(Z)=0$. Then $$H^2\Big(Z',T_{Z'}\big(-\log(\sum_{i=1}^n (F_i+H_i) + \sum_{i=1}^m G_i \big)\Big)= m+n-2.$$
\label{obstr=}
\end{lemma}

\begin{proof}
First, by the $(-1)$ principle, we have $$H^2\big(Z',T_{Z'}\big(-\log(\sum_{i=1}^n (F_i+H_i) \big)\big)=H^2\big(Z',T_{Z'}\big(-\log(\sum_{i=1}^n F_i \big)\big).$$ We have the inclusions $$H^0(\Omega_{Z'}^1\big(\log \sum_{i=1}^n F_i \big) \otimes K_{Z'}) \subset H^0(\Omega_{Z'}^1\big(\log \sum_{i=2}^n F_i \big) \otimes K_{Z'}+F_1) \ \ \ \ \ \ \ \ $$ $$ \ \ \ \ \ \ \ \  \subset H^0(\Omega_{Z'}^1\big(\log \sum_{i=2}^n F_i \big) \otimes \sum_{i=2}^n H_i ) \subset H^0(\Omega_{Z'}^1\big(\log \sum_{i=2}^n F_i + \sum_{i=2}^n H_i \big)),$$ and by the long exact sequence in cohomology applied to the residue sequence $$ 0 \to \Omega_{Z'}^1 \to \Omega_{Z'}^1\big(\log \sum_{i=2}^n F_i + \sum_{i=2}^n H_i \big) \to \bigoplus \O_{F_i} \oplus     \bigoplus \O_{H_i}\to 0, $$ we obtain that $H^0(\Omega_{Z'}^1\big(\log \sum_{i=2}^n F_i + \sum_{i=2}^n H_i \big))$ is the kernel of the Chern map $$H^0(\O_{F_i} \oplus \O_{H_i}) \to H^1(\Omega_{Z'}^1).$$ This kernel has dimension $n-2$, as we have $n-1$ fibers with two components each. 

By the (-2) principle, we can add the configurations $G'_i$ which are the $G_i$ minus one $(-2)$-curve without altering the dimension, so $$h^0(\Omega_{Z'}^1\big(\log \sum_{i=1}^n F_i + \sum_{i=1}^m G'_i \big) \otimes K_{Z'}) \leq n-2,$$ and by the (-1) principle we can add the $H_i$, so  $$h^0(\Omega_{Z'}^1\big(\log \sum_{i=1}^n F_i + \sum_{i=1}^n H_i + \sum_{i=1}^m G'_i \big) \otimes K_{Z'}) \leq n-2.$$ 

As the normal bundle of a $(-2)$-curve has a one-dimensional first cohomology, when we add the missing $(-2)$-curves from the $G_i$, the dimension increases by at most $1$ for each one, so $$h^0(\Omega_{Z'}^1\big(\log \sum_{i=1}^n F_i + \sum_{i=1}^n H_i + \sum_{i=1}^m G_i \big) \otimes K_{Z'}) \leq m+n-2.$$ Finally, by Lemma \ref{ineq}, this dimension is also greater than $m+n-2$, so the claim follows. 
\end{proof}

In what follows, we will again use the fact that adding a transversal $(-2)$-curve to a SNC configuration increases the dimension by at most $1$. As in the previous proof, this claim follows from the standard residue sequences (see \cite[Section 4]{PSU13}). 

\section{Exotic $\C \P^2 \# 4 \overline{\C \P^2}$} \label{s5}

There are several constructions of exotic $\C \P^2 \# 5 \overline{\C \P^2}$ (see \cites{A08,ABP08,FS11,BK17,BKS22}), but the only examples in the literature obtained via the rational blowdown technique are in \cite{PPS09b} --where the authors found two examples that admit a complex structure, from distinct Wahl singularities-- and \cite{BKS22}. It turns out that it is possible to construct \textbf{hundreds} of pairs of Wahl chains in suitable rational surfaces which produce exotic $\C \P^2 \# 5 \overline{\C \P^2}$ by means of Theorem \ref{min} and Corollary \ref{exotic}. Among them, we have found examples with a Wahl chain of length $r$ for each $1\leq r \leq 18$, in particular, the bound $4K_W^2+1$ on lengths for nonrational surfaces does not hold in the rational case. Typically $H^2(T_W)=0$ for these surfaces, so they admit a complex structure. The local dimension at $W$ for the KSBA compactification of the moduli space is $2$, so these examples correspond to hundreds of KSBA boundary curves. 

\begin{Question}
Is the moduli space of simply-connected $p_g=0$ surfaces of general type with $K^2=4$ an irreducible surface? Is it a ruled surface? 
\end{Question}

We will describe these examples with $K^2 = 4$ and connections between them somewhere else. In this section we will focus on constructions for the $K^2=5$ case, i.e. exotic $\C \P^2 \# 4 \overline{\C \P^2}$. 

\vspace{0.5cm}

Let us consider the pencil of cubics in $\P^2_{\C}$ $$\{\mu xyz + \nu (x-z)(x+y+z)(x-y+z)=0 \colon \ [\mu,\nu] \in \P^1_{\C} \}.$$ Its base points are $[0,1,0]$ (double), $[-1,0,1]$ (double), $[0,1,1]$, $[0,-1,1]$, $[1,-1,0]$, and $[1,1,0]$. Let us name the lines $L_x = \{x =0\}$, $L_y = \{y=0\}$, $L_z = \{z=0\}$, $A = \{x - z=0\}$, $B= \{x+y+z=0\}$, and $C = \{x - y + z=0\}$. The blow-up at the nine base points defines an elliptic fibration $f \colon Z \to \P^1_{\C}$ with precisely four singular fibers: two of type $I_4$ and two of type $I_2$. Let us denote the exceptional curves of the blow-up as: $E_1,E_2$ for $[0,1,0]$, $E_3,E_4$ for $[-1,0,1]$, $E_6$ for $[0,1,1]$, $E_7$ for $[0,-1,1]$, $E_8$ for $[1,1,0]$, and $E_9$ for $[1,-1,0]$. Consider the curves $Q_1 = \{ (x + z)^2 - y(x - z)=0 \}$, $L_1 = \{x + y - z=0 \}$, $Q_2 = \{(x + z)^2 + y(x - z)=0\}$, and $L_2 = \{x - y - z=0\}$. We have that $$Q_1 \cap L_1= \{B_1:=[i,-i+1,1], T_1:=[-i,i+1,1] \}, \quad \text{and}$$ $$Q_2 \cap L_2= \{B_2:=[i,i-1,1], T_2:=[-i,-i-1,1] \},$$ where $i^2=-1$. The singular fibers of $f$ are $$E_1+L_z+L_y+L_x, \quad E_3+B+A+C, \quad Q_1+L_1, \quad \text{and}  \quad Q_2+L_2.$$ The fibration $f$ has precisely eight disjoint sections: $E_2,E_4,E_5,E_6,E_7,E_8,E_9$, and $H$, the last being the strict transform of the line $\{x+z=0\}$. We will denote by the same symbols the strict transforms of curves under birational morphisms. 

We also consider the following double sections of $f$:

\begin{itemize}
    \item $BT:= \{ x-iy+z=0\}$, it passes through $[-1,0,1]$, $B_1$, $T_2$.
    \item $TB:= \{ x+iy+z=0\}$, it passes through $[-1,0,1]$, $T_1$, $B_2$.
    \item $BB:= \{ x-iz=0\}$, it passes through $[0,1,0]$, $B_1$, $B_2$.
    \item $TT:= \{ x+iz=0\}$, it passes through $[0,1,0]$, $T_1$, $T_2$.
    \item $N:= \{ (2i-1)x-y-z=0\}$, it passes through $[0,-1,1]$, $T_1$.
\end{itemize}

In Figure \ref{f1}, we show how these curves intersect in $Z$. This is a simple crossings configuration $D$ of $25$ curves with 12 $(-2)$-curves, 8 $(-1)$-curves, and 5 $(0)$-curves (double sections). Moreover $t_2=62$, $t_4=3$, and $t_5=1$ and $t_k=0$ for other $k$. In this way, by Proposition \ref{peo}, we have $\bar c_1^2=82$, $\bar c_2=37$, and so $\bar c_1^2 / \bar c_2 = 2. \overline{216}$. 

\begin{figure}[htbp]
\centering
\includegraphics[width=12.7cm]{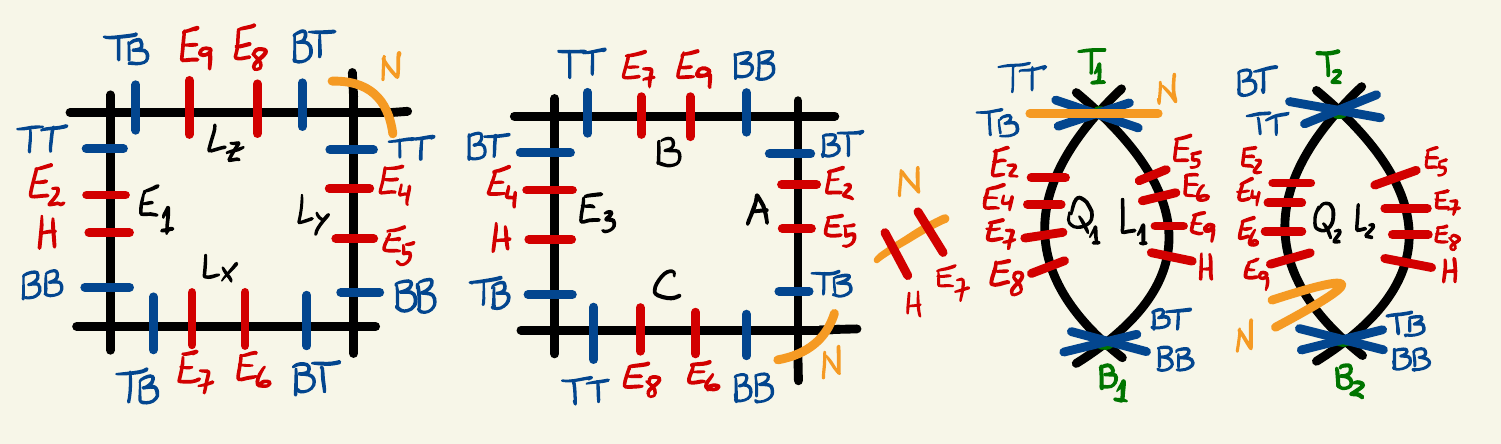}
\caption{The $25$ relevant $\P^1_{\C}$ in $Z$} \label{f1}
\end{figure}

Below we will consider subconfigurations $D_0$ from $D$ to construct surfaces $W$ with two Wahl singularities and $K_W$ big and nef, so that the corresponding rational blowdown is a minimal symplectic exotic $\C \P^2 \# 4 \overline{\C \P^2}$. For this and the next section, we adopt the following way to express the information necessary to construct $X$ (that is, a choice of blow-ups of $Z$ at points in $D_0$), and $W$ (contraction of two Wahl chains in $X$). For this, it indicates $K_W^2$, $D_0$, the points to be blown-up, and the Wahl chains together with their $(n,a)$:

\begin{framed}

\noindent $K_W^2$ - configuration $D_0$ - points where blow-ups happen and in the order shown - $(n_1,a_1):[b_1,\ldots,b_{\ell_1}]$ - $(n_2,a_2):[c_1,\ldots,c_{\ell_2}]$

\end{framed}

There are two different notations for points where blow-ups happen. If there is a single blow-up at an intersection of curves $L$ and $R$, then it is denoted by $L \cap R$. If there are several infinitely near blow-ups over that intersection, then it is denoted by $[a_1,\ldots,a_n] \times L \cap R$, where $n$ blow-ups are done at successive intersections, with a $(-a_1)$-curve intersecting $L$ and a $(-a_n)$-curve intersecting $R$. 

If there is a $k$-point, then we denote its blow-up by $L_1 \cap \ldots \cap L_k$, where $L_i$ are the curves passing through that point, counted with multiplicities. If another blow-up is realized over the exceptional curve, then we also give a name to it as $E := L_1 \cap \ldots \cap L_k$.

We will elaborate on all details for the first example, and then we only comment on particular characteristics for the rest. 

\begin{framed}

\textbf{(1)} $K_W^2=5$ - $\{L_z,~\allowbreak L_y,~\allowbreak L_x,~\allowbreak E_1,~\allowbreak B,~\allowbreak C,~\allowbreak E_3,~\allowbreak Q_1,~\allowbreak L_1,~\allowbreak Q_2,~\allowbreak L_2,~\allowbreak E_5,~\allowbreak E_6,~\allowbreak BT\}$ - $~E_{17} := Q_1 \cap L_1 \cap BT$, $~Q_2 \cap L_2 \cap BT$, $~L_z \cap E_1$, $~L_y \cap E_5$, $~L_x \cap E_6$, $~L_x \cap BT$, $~E_3 \cap BT$, $~L_1 \cap E_6$, $~Q_1 \cap E_{17}$, $[2,1] \times ~Q_1 \cap L_1$ - $(3,1) : [5,\allowbreak 2]$, $(700,257) : [3,\allowbreak 4,\allowbreak 3,\allowbreak 3,\allowbreak 4,\allowbreak 2,\allowbreak 6,\allowbreak 2,\allowbreak 3,\allowbreak 3,\allowbreak 3,\allowbreak 2,\allowbreak 3,\allowbreak 2]$

\end{framed}

Let us consider the subconfiguration $$ D_0:= \{L_z,~\allowbreak L_y,~\allowbreak L_x,~\allowbreak E_1,~\allowbreak B,~\allowbreak C,~\allowbreak E_3,~\allowbreak Q_1,~\allowbreak L_1,~\allowbreak Q_2,~\allowbreak L_2,~\allowbreak E_5,~\allowbreak E_6,~\allowbreak BT\}.$$ It has $r=14$ curves: 11 $(-2)$-curves, 2 $(-1)$-curves, and 1 $(0)$-curve. It also has $t_2=18$ and $t_3=2$, and so $\bar c_1^2=14$ and $\bar c_2=6$. We now consider particular blow-ups at singular points of $D_0$, as indicated in Figure \ref{f2}. 

\begin{figure}[htbp]
\centering
\includegraphics[width=12.7cm]{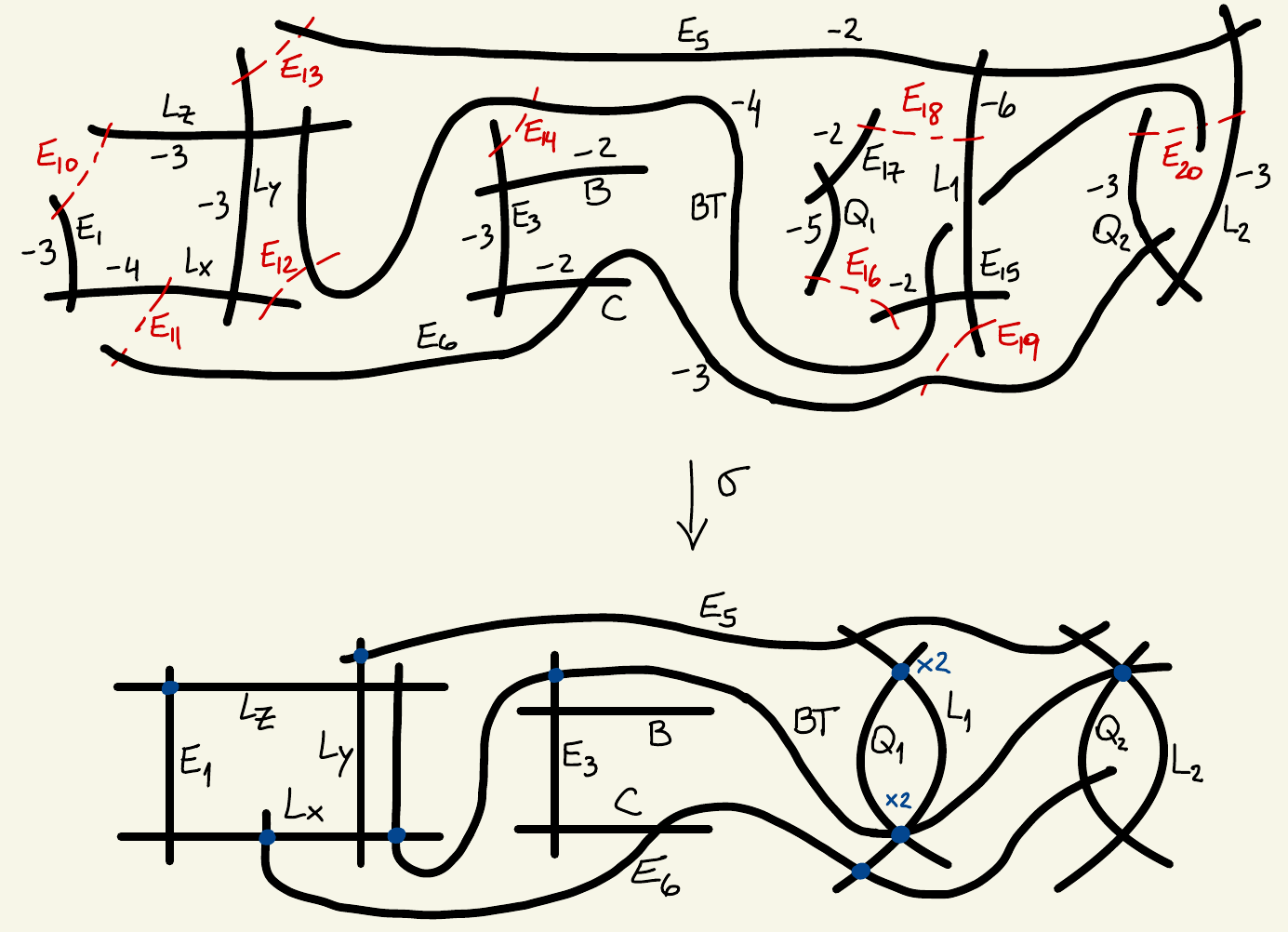}
\caption{The blow-up $\sigma \colon X \to Z$} \label{f2}
\end{figure}

In $X$ we can find two Wahl chains corresponding to the Hirzebruch-Jung continued fractions $[5,\allowbreak 2]$, and $[3,\allowbreak 4,\allowbreak 3,\allowbreak 3,\allowbreak 4,\allowbreak 2,\allowbreak 6,\allowbreak 2,\allowbreak 3,\allowbreak 3,\allowbreak 3,\allowbreak 2,\allowbreak 3,\allowbreak 2]$. Let $\pi \colon X \to W$ the contraction of these two Wahl chains, so $W$ is a normal projective surface with two Wahl singularities \cite{Art62}.

\begin{proposition}
The canonical divisor $K_W$ is ample, and $K_W^2=5$.
\end{proposition}

\begin{proof}
We have $K_X^2+2+14=K_W^2$, $K_X^2=K_Z^2-11=-11$, so $K_W^2=5$.

Moreover, $K_X \sim \sigma^*(-F)+ \sum_{i=10}^{20} E_i + E_{16}+E_{18}$, where $F$ is a fiber of $Z \to \P_{\C}^1$. We note that $$ -\frac{Q_1}{2} - \frac{L_1}{2} - E_{17} -\frac{3 E_{18}}{2} - E_{15} - \frac{3 E_{16}}{2} - \frac{E_{19}}{2} \equiv \sigma^*\left(-\frac{F}{2}\right) \equiv -\frac{Q_2}{2} - \frac{L_2}{2} - E_{20},$$ and so
$$K_X \equiv -\frac{Q_1}{2} - \frac{L_1}{2} -\frac{Q_2}{2} - \frac{L_2}{2} + \sum_{i=10}^{14} E_i  + \frac{E_{16}}{2}  + \frac{E_{18}}{2}  + \frac{E_{19}}{2}.$$
On the other hand, we have \begin{align*}
    \pi^*(K_W) &\equiv K_X + \frac{2}{3} Q_1 + \frac{1}{3} E_{17} + \frac{443}{700} E_1 + \frac{629}{700} L_x + \frac{673}{700} L_y + \frac{690}{700} L_z + \frac{697}{700} BT\\
    &\quad +\frac{698}{700} E_{15} +\frac{699}{700} L_1 + \frac{696}{700} E_5 + \frac{693}{700} L_2 + \frac{683}{700} Q_2 + \frac{656}{700} E_6 + \frac{585}{700} C\\
    &\quad + \frac{514}{700} E_3 + \frac{257}{700} B,
\end{align*} and so $\pi^*(K_W)$ can be written as a nonnegative rational combination of curves. By the Nakai-Moishezon criterion, for $K_W$ to be ample we need to check that for each irreducible curve $\Gamma$ in $X$, not exceptional for $\pi$, we have that $\sigma^*(K_W) \cdot \Gamma >0$. Hence this condition needs to be checked for curves in the support of $\sigma^*(K_W)$ and we also must verify that there are no curves disjoint from that support. One easily verifies the first explicitly. For the second, we point out that this support of $\sigma^*(K_W)$ contains the curves $L_1, E_{15}, E_{16}, Q_1, E_{17}, E_{18}, E_{19}$, and so a $\Gamma$ disjoint from this support must be contained in a fiber of $X \to \P_{\C}^1$. But this is not possible in our example. Therefore $K_W$ is ample. 

\end{proof}

\begin{theorem}
The rational blowdown of $X$ at these two Wahl chains is an exotic (minimal, symplectic) $\C \P^2 \# 4 \overline{\C \P^2}$.
\end{theorem}

\begin{proof}
As $K_W$ is ample by the previous proposition, by Theorem \ref{min}, the rational blowdown $Y$ of $\exc(\pi)$ is a minimal symplectic 4-manifold. By Corollary \ref{exotic}, $Y$ is an exotic $\C \P^2 \# 4 \overline{\C \P^2}$ provided that the intersection form is odd and $Y$ is simply-connected. The intersection form is indeed odd as $K_Y^2=5$. For the fundamental group we apply an analogous computation with Seifert-Van Kampen's theorem as in \cite{LP07}: Consider some simple loops $\alpha$ around $E_{17}$, and $\beta$ around $L_1$. They are homotopic via the sphere $E_{18}$. As gcd$(3,700)=1$, we see that $\alpha$ and $\beta$ are trivial in the complement of the Wahl chains. As $\alpha$ is a generator for the complement of $Q_1 \cup E_{17}$, the only thing left is to check that a generator for the other Wahl chain is also trivial. Let $\gamma$ be the generator defined by a loop around $E_1$. Then, following Mumford's computation \cite{Mum61}, we obtain that $\beta$ and $\gamma^{493}$ are homotopic. But gcd$(493,700)=1$, so $\beta$ is also a generator, but we already saw that it is trivial in the   complement of the Wahl chains.
\end{proof}

\begin{proposition}
The surface $W$ satisfies $H^1(W,T_W)=0$ and $H^2(T_W)=\C^2$.
\label{obtrex}
\end{proposition}

\begin{proof}
To compute $H^2(T_W)$, we first start with the elliptic fibration $Z \to \P_{\C}^1$ and consider the singular fibers $E_1+L_x+L_y+L_z$, $Q_1+L_1$, and $Q_2+L_2$. We have that $H^2(Z,T_Z(-\log (E_1+L_x+L_y+L_z+Q_1+L_1+Q_2+L_2)))=\C$ by Lemma \ref{obstr=}. By the (-2) principle we can add the curves $C+E_3+B$ keeping this dimension; by the (-1) principle we can add $E_5+E_6$ keeping again the dimension. We now blow-up $B_1$ and $T_2$. The curve $BT$ becomes a $(-2)$-curve. Thus when we add $BT$ we may keep the dimension or increase it by $1$. By Theorem \ref{obstr}, we have that $h^2(T_W)=h^1(T_W)+2$, and so $h^2(T_W) \geq 2$. Therefore $h^2(T_W)=2$ (in particular, the addition of $BT$ indeed increases the dimension by $1$), and $h^1(T_W)=0$.
\end{proof}

In this case, we have the exact sequence $$0 \to  \T^1_{QG,W} \to H^0(W, \mathcal{T}^1_{QG,W}) \simeq \C^2 \to H^2(W,T_W) \simeq \C^2 \to \T^2_{QG,W} \to 0,$$ and $\T^1_{QG,W}=\T^2_{QG,W}=\C^m$ with $m=0,1$ or $2$. We do not know if $W$ admits a $\Q$-Gorenstein smoothing. If it did, then it would be the first example of a simply-connected surface of general type with $p_g=0$ and $K^2=5$. 
\bigskip 

The following examples also define exotic (minimal, symplectic) $\C \P^2 \# 4 \overline{\C \P^2}$ from surfaces with two Wahl chains, one of them equal to $[3,\allowbreak 4,\allowbreak 3,\allowbreak 3,\allowbreak 4,\allowbreak 2,\allowbreak 6,\allowbreak 2,\allowbreak 3,\allowbreak 3,\allowbreak 3,\allowbreak 2,\allowbreak 3,\allowbreak 2]$. Those examples suggest that there should be at least a one dimensional equisingular family of such surfaces.     

\begin{framed}

\textbf{(2)} $K^2=5$ - $\{L_Z,~\allowbreak L_X,~\allowbreak B,~\allowbreak E_3,~\allowbreak Q_1,~\allowbreak L_1,~\allowbreak Q_2,~\allowbreak L_2,~\allowbreak E_2,~\allowbreak E_5,~\allowbreak E_8,~\allowbreak BT,~\allowbreak TB,~\allowbreak BB,~\allowbreak TT,~\allowbreak N\}$ - $~Q_1 \cap L_1 \cap TB \cap TT \cap N$, $~B_1 := Q_1 \cap L_1 \cap BT \cap BB$, $~Q_2 \cap L_2 \cap BT \cap TT$, $~Q_2 \cap L_2 \cap TB \cap BB$, $~L_Z \cap E_8$, $~L_Z \cap BT$, $~B \cap E_3$, $~Q_1 \cap B_1$, $~Q_2 \cap N$, $~B_1 \cap BT$, $~L_X \cap TB$, $[2,\allowbreak 3,\allowbreak 3,\allowbreak 3,\allowbreak 2,\allowbreak 3,\allowbreak 2,\allowbreak 1,\allowbreak 3,\allowbreak 4,\allowbreak 3,\allowbreak 3] \times ~Q_2 \cap E_2$ - $(700,257) : [3,\allowbreak 4,\allowbreak 3,\allowbreak 3,\allowbreak 4,\allowbreak 2,\allowbreak 6,\allowbreak 2,\allowbreak 3,\allowbreak 3,\allowbreak 3,\allowbreak 2,\allowbreak 3,\allowbreak 2]$, $(493,181) : [3,\allowbreak 4,\allowbreak 3,\allowbreak 3,\allowbreak 4,\allowbreak 5,\allowbreak 2,\allowbreak 3,\allowbreak 3,\allowbreak 3,\allowbreak 2,\allowbreak 3,\allowbreak 2]$

\end{framed}


\begin{framed}

\textbf{(3)} $K^2=5$ - $\{L_Z,~\allowbreak L_Y,~\allowbreak L_X,~\allowbreak E_1,~\allowbreak B,~\allowbreak C,~\allowbreak E_3,~\allowbreak Q_1,~\allowbreak L_1,~\allowbreak Q_2,~\allowbreak L_2,~\allowbreak E_5,~\allowbreak E_6,~\allowbreak BT\}$ - $~B_1 := Q_1 \cap L_1 \cap BT$, $~Q_2 \cap L_2 \cap BT$, $~L_Z \cap E_1$, $~L_Y \cap E_5$, $~L_X \cap E_6$, $~L_X \cap BT$, $~E_3 \cap BT$, $~L_1 \cap E_6$, $~L_1 \cap B_1$, $[2,\allowbreak 3,\allowbreak 3,\allowbreak 3,\allowbreak 2,\allowbreak 3,\allowbreak 2,\allowbreak 1,\allowbreak 3,\allowbreak 4,\allowbreak 3,\allowbreak 3,\allowbreak 4,\allowbreak 2] \times ~Q_1 \cap L_1$ - $(700,257) : [3,\allowbreak 4,\allowbreak 3,\allowbreak 3,\allowbreak 4,\allowbreak 2,\allowbreak 6,\allowbreak 2,\allowbreak 3,\allowbreak 3,\allowbreak 3,\allowbreak 2,\allowbreak 3,\allowbreak 2]$, $(700,257) : [3,\allowbreak 4,\allowbreak 3,\allowbreak 3,\allowbreak 4,\allowbreak 2,\allowbreak 6,\allowbreak 2,\allowbreak 3,\allowbreak 3,\allowbreak 3,\allowbreak 2,\allowbreak 3,\allowbreak 2]$

\end{framed}


\bigskip

Moreover, we have found 5 other examples from other subconfigurations of $D$.

\begin{framed}

\textbf{(4)} $K^2=5$ - $\{L_Z,~\allowbreak L_Y,~\allowbreak L_X,~\allowbreak B,~\allowbreak C,~\allowbreak E_3,~\allowbreak Q_1,~\allowbreak L_1,~\allowbreak Q_2,~\allowbreak L_2,~\allowbreak E_2,~\allowbreak E_5,~\allowbreak E_6,~\allowbreak BT,~\allowbreak BB\}$ - $~Q_1 \cap L_1 \cap BT \cap BB$, $~Q_2 \cap L_2 \cap BT$, $~Q_2 \cap L_2 \cap BB$, $~L_Y \cap E_5$, $~L_Y \cap BB$, $~L_X \cap BT$, $~L_1 \cap E_6$, $~C \cap BB$, $[2,\allowbreak 2,1] \times ~BT \cap E_3$, $[2,\allowbreak 2,1] \times ~C \cap E_6$ - $(256,75) : [4,\allowbreak 2,\allowbreak 4,\allowbreak 2,\allowbreak 3,\allowbreak 5,\allowbreak 3,\allowbreak 4,\allowbreak 2,\allowbreak 4,\allowbreak 2,\allowbreak 2]$, $(17,5) : [4,\allowbreak 2,\allowbreak 5,\allowbreak 4,\allowbreak 2,\allowbreak 2]$

\end{framed}


\begin{framed}

\textbf{(5)} $K^2=5$ - $\{L_Y,~\allowbreak E_1,~\allowbreak A,~\allowbreak E_3,~\allowbreak Q_1,~\allowbreak L_1,~\allowbreak Q_2,~\allowbreak L_2,~\allowbreak E_6,~\allowbreak E_7,~\allowbreak E_8,~\allowbreak BT,~\allowbreak TB,~\allowbreak BB,~\allowbreak TT,~\allowbreak N\}$ - $~Q_1 \cap L_1 \cap TB \cap TT \cap N$, $~B_1 := Q_1 \cap L_1 \cap BT \cap BB$, $~Q_2 \cap L_2 \cap BT \cap TT$, $~Q_2 \cap L_2 \cap TB \cap BB$, $~L_Y \cap N$, $~A \cap BT$, $~Q_1 \cap E_7$, $~B_1 \cap BT$, $~L_Y \cap BB$, $~Q_2 \cap N$, $[2,\allowbreak 2,1] \times ~Q_1 \cap B_1$, $[2,\allowbreak 2,1] \times ~Q_2 \cap N$ - $(82,25) : [4,\allowbreak 2,\allowbreak 2,\allowbreak 3,\allowbreak 6,\allowbreak 2,\allowbreak 3,\allowbreak 5,\allowbreak 2,\allowbreak 2]$, $(59,18) : [4,\allowbreak 2,\allowbreak 2,\allowbreak 3,\allowbreak 5,\allowbreak 3,\allowbreak 5,\allowbreak 2,\allowbreak 2]$

\end{framed}

\begin{framed}

\textbf{(6)} $K^2=5$ - $\{L_Z,~\allowbreak L_X,~\allowbreak E_1,~\allowbreak B,~\allowbreak C,~\allowbreak E_3,~\allowbreak Q_1,~\allowbreak L_1,~\allowbreak Q_2,~\allowbreak L_2,~\allowbreak E_2,~\allowbreak E_5,~\allowbreak TB,~\allowbreak TT,~\allowbreak N\}$ - $~Q_1 \cap L_1 \cap TB \cap TT \cap N$, $~Q_2 \cap L_2 \cap TT$, $~Q_2 \cap L_2 \cap TB$, $~L_Z \cap N$, $~E_1 \cap E_2$, $~C \cap N$, $~E_3 \cap TB$, $~B \cap TT$, $[2,1] \times ~L_X \cap E_1$, $[2,1] \times ~N \cap Q_2$ - $(89,34) : [3,\allowbreak 3,\allowbreak 3,\allowbreak 3,\allowbreak 5,\allowbreak 3,\allowbreak 3,\allowbreak 3,\allowbreak 2]$, $(26,11) : [3,\allowbreak 2,\allowbreak 3,\allowbreak 2,\allowbreak 6,\allowbreak 4,\allowbreak 2]$

\end{framed}


\begin{framed}

\textbf{(7)} $K^2=5$ - $\{L_Z,~\allowbreak L_X,~\allowbreak E_1,~\allowbreak B,~\allowbreak C,~\allowbreak E_3,~\allowbreak Q_1,~\allowbreak L_1,~\allowbreak Q_2,~\allowbreak L_2,~\allowbreak E_5,~\allowbreak TB,~\allowbreak BB,~\allowbreak N\}$ - $~Q_1 \cap L_1 \cap TB \cap N$, $~B_1 := Q_1 \cap L_1 \cap BB$, $~Q_2 \cap L_2 \cap TB \cap BB$, $~L_Z \cap N$, $~L_X \cap E_1$, $~C \cap N$, $~E_3 \cap TB$, $~B_1 \cap BB$, $[2,1] \times ~B \cap BB$, $[2,\allowbreak 2,1] \times ~N \cap Q_2$ - $(89,34) : [3,\allowbreak 3,\allowbreak 3,\allowbreak 3,\allowbreak 5,\allowbreak 3,\allowbreak 3,\allowbreak 3,\allowbreak 2]$, $(37,11) : [4,\allowbreak 2,\allowbreak 3,\allowbreak 2,\allowbreak 6,\allowbreak 4,\allowbreak 2,\allowbreak 2]$

\end{framed}


\begin{framed}

\textbf{(8)} $K^2=5$ - $\{L_Z,~\allowbreak L_X,~\allowbreak E_1,~\allowbreak B,~\allowbreak C,~\allowbreak E_3,~\allowbreak Q_1,~\allowbreak L_1,~\allowbreak Q_2,~\allowbreak L_2,~\allowbreak E_2,~\allowbreak E_5,~\allowbreak TB,~\allowbreak TT,~\allowbreak N\}$ - $~Q_1 \cap L_1 \cap TB \cap TT \cap N$, $~Q_2 \cap L_2 \cap TT$, $~Q_2 \cap L_2 \cap TB$, $~L_Z \cap N$, $~E_1 \cap E_2$, $~E_1 \cap TT$, $~C \cap E_3$, $~C \cap N$, $[2,\allowbreak 2,1] \times ~L_X \cap TB$, $[2,1] \times ~N \cap Q_2$ - $(111,31) : [4,\allowbreak 3,\allowbreak 2,\allowbreak 3,\allowbreak 5,\allowbreak 3,\allowbreak 4,\allowbreak 3,\allowbreak 2,\allowbreak 2]$, $(26,11) : [3,\allowbreak 2,\allowbreak 3,\allowbreak 2,\allowbreak 6,\allowbreak 4,\allowbreak 2]$

\end{framed}


\section{Exotic $3 \C \P^2 \# b^- \overline{\C \P^2}$ for $b^-=9,8,7$} \label{s6}

In \cite{RU21}, we constructed exotic $3 \C \P^2 \# b^- \overline{\C \P^2}$ for $b^-=18,17,\ldots,10$ (see also \cite{PPS13}). They are actually (simply-connected) complex surfaces of general type with $K^2=19-b^-$, and $p_g=1$. In this section, we show how to go below $b^-=10$, providing several examples for $b^-=9,8,7$. We do not see any constraints to getting below $7$, but we have not yet been able to produce those examples. In addition, we can find many surfaces $W$ for $b^- \geq 10$ (thousands), but very few for $b^-<10$. As a matter of fact, we have found thousands of surfaces $W$ for $b^- \geq 10$, but only a few for $b^- < 10$. We note that there exist exotic examples for $b^-=6,5,4$ by other means (see e.g. \cites{AP08,ABP08,AP10,BK17}), but this seems to be an open question for $b^- \leq 3$.
\bigskip

All of our examples will be constructed from a single nodal simple crossings configuration $D$ in a particular K3 surface $Z$, which admits an elliptic fibration $Z \to \P_{\C}^1$ with singular fibers of type $2I_8+2I_1+I_2+I_4$. We will construct $Z$ via a degree two base change of a rational elliptic fibration $Z_0 \to \P_\C^1$ with singular fibers $I_8+2I_1+I_2$, branched at an $I_1$ and the $I_2$. the configuration $D$ will then consist of $36$ rational curves in $Z$.

Let us consider the pencil of cubics in $\P^2_{\C}$ $$\{4\mu xyz + \nu (z-y)(x^2-yz)=0 \colon \ [\mu,\nu] \in \P^1_{\C} \}.$$ Its base points consist of $[0,0,1]$ (triple), $[1,0,0]$ (double), $[0,1,1]$, and $[0,1,0]$ (triple). Let us name: $L_1 = \{y =0\}$, $L_2 = \{x=0\}$, $L_3 = \{z=0\}$, $C = \{x^2-yz=0\}$, and $L= \{z-y=0\}$. The blow-up at the nine base points defines an elliptic fibration $f_0 \colon Z_0 \to \P^1_{\C}$. Denote the exceptional curves of the blow-up as: $E_1,E_2,E_3$ for $[0,0,1]$, $E_4,E_5$ for $[1,0,0]$, $E_6$ for $[0,1,1]$, and $E_7, E_8, E_9$ for $[0,1,0]$. The elliptic fibration $f_0$ has singular fibers:

\begin{itemize}
\item[$I_8$:] $L_2+E_7+E_8+L_3+E_4+L_1+E_2+E_1$,

\item[$I_2$:] $C+L$,

\item[$I_1$:] $F_1$ defined by $4xyz+(z-y)(x^2-yz)=0$ with node at $[1,-1,1]$,

\item[$I_1$:] $F_2$ defined by $4xyz-(z-y)(x^2-yz)=0$ with node at $[1,1,-1]$. 
\end{itemize}

It has four sections: $E_3,E_6,E_5,E_9$. Let us also consider the double sections of $f_0$ (which are $0$-curves): $M$: $z+y=0$, $N$: $x^2+yz=0$, $MR_1$: $x-y=0$, $MR_2$: $x+y=0$, $ML_1$: $x-z=0$, and $ML_2$: $x+z=0$. The intersections among all of these curves are shown in Figure \ref{f3}.

\begin{figure}[htbp]
\centering
\includegraphics[width=12.7cm]{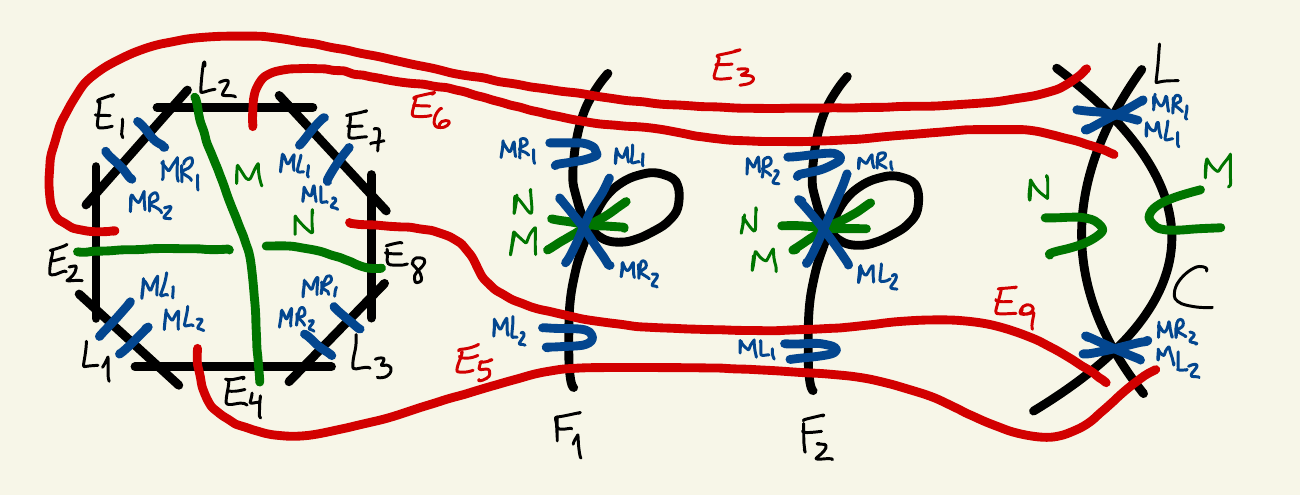}
\caption{A configuration of 22 rational curves in $Z_0$} \label{f3}
\end{figure}

Let $D_0$ be the configuration formed by the 12 singular fibers, the 4 sections, and the 6 double sections defined above. It is a nodal simple crossings configuration with $r=22$, $\nu=2$, $t_2=48$, $t_4=4$, $t_6=2$, and $t_k=0$ for all other $k$. We compute $\bar c_1^2=96$, $\bar c_2=38$, and so $\bar c_1^2 / \bar c_2=48/19 \approx 2.53$.

Let us consider the base change of $f_0$ of order $2$ branched over the singular fibers $F_2$ and $C+L$. Let $g \colon Z \to Z_0$ be the corresponding composition of the double cover and the minimal resolution over the nodes at $F_2$ and at $L+C$. Then we obtain an elliptic fibration $f \colon Z \to Z_0$ (where $Z$ is a K3 surface) with singular fibers 2 $I_8$ (pre-image of $I_8$), $I_2$ (pre-image of $F_2$), $I_4$ (pre-image of $L+C$), and 2 $I_1$ (pre-image of $F_1$). Let $D$ be the pre-image of $D_0$ under $g$, still a nodal simple crossings configuration. In Figure \ref{f4}, we give notations for all the 36 rational curves in $D$ and we show how they intersect.

\begin{figure}[htbp]
\centering
\includegraphics[width=12.7cm]{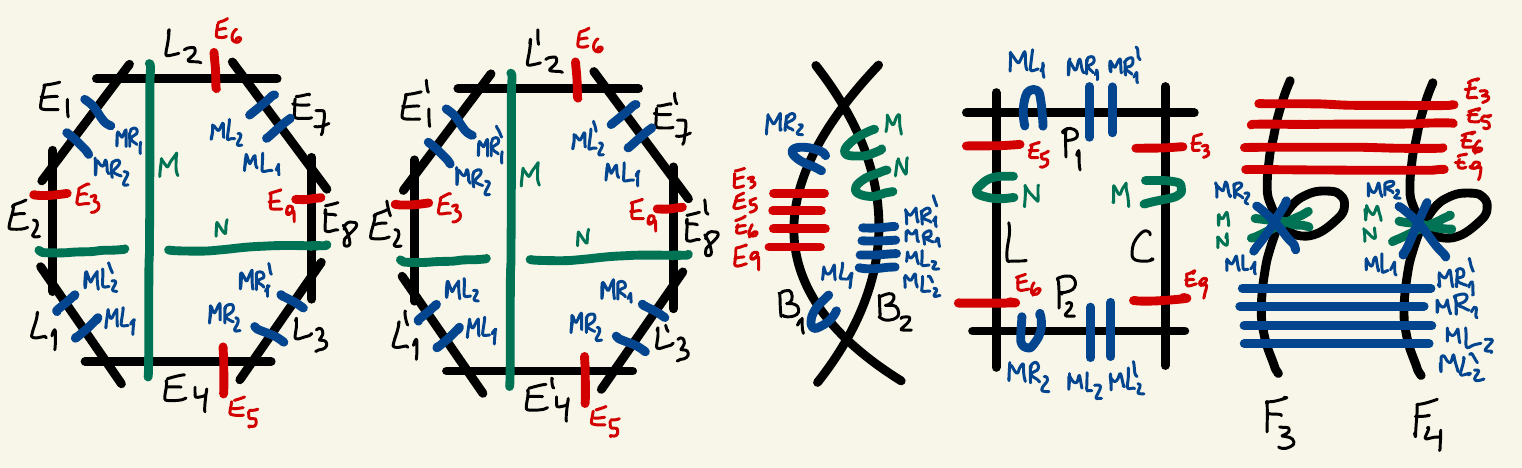}
\caption{A configuration of 36 rational curves in $Z$} \label{f4}
\end{figure}

Note that the pre-image of $MR_1$ and $ML_2$ split into two sections each. We denote those pre-images by $MR_1, MR_1'$ and $ML_2, ML_2'$ respectively. A similar convention is used for components of the $I_8$'s fibers. We use $F_3$ and $F_4$ for the pre-image of $F_1$. The rest of the curves keep the same names. The configuration $D$ has $r=36$, $\nu=2$, $t_2=102$, $t_6=2$, and $t_k=0$ for other $k$. The Chern numbers of $(Z,D)$ are $\bar c_1^2=160$, and $\bar c_2=64$, so $\bar c_1^2/\bar c_2=2.5$. 

In what follows, we will show subconfigurations of $D$ which produce the claimed examples for $K^2=10,11,12$. As in the previous section, we will present each of them by means of a box with the necessary information to construct it. We will develop details for some of them. 

\begin{center}
\bigskip 
\textbf{$K^2=10$}
\bigskip
\end{center}

\begin{framed}

\textbf{(9)} $K^2=10$ - $\{L_2,~\allowbreak E_7,~\allowbreak E_8,~\allowbreak L_3,~\allowbreak E_4,~\allowbreak L_1,~\allowbreak E_2,~\allowbreak E_1,~\allowbreak L'_2,~\allowbreak E'_7,~\allowbreak E'_8,~\allowbreak L'_3,~\allowbreak E'_4,~\allowbreak L'_1,~\allowbreak E'_2,~\allowbreak E'_1,~\allowbreak F_3,~\allowbreak E_3,~\allowbreak E_9,~\allowbreak ML_1,~\allowbreak MR_1\}$ - $~F_3 \cap F_3 \cap ML_1$, $~E_7 \cap ML_1$, $~E_8 \cap E_9$, $~L_1 \cap ML_1$, $~E_2 \cap E_1$, $~E'_8 \cap L'_3$, $~L'_1 \cap E'_2$, $~F_3 \cap E_3$, $~E_1 \cap MR_1$, $[2,\allowbreak 2,1] \times ~E'_8 \cap E'_7$, $[3,\allowbreak 5,\allowbreak 3,\allowbreak 4,\allowbreak 2,\allowbreak 2,\allowbreak 1,\allowbreak 4,\allowbreak 2,\allowbreak 3,\allowbreak 3,\allowbreak 2,\allowbreak 2,\allowbreak 3] \times ~E'_4 \cap L'_3$ - $(19843,5873) : [4,\allowbreak 2,\allowbreak 3,\allowbreak 3,\allowbreak 2,\allowbreak 2,\allowbreak 3,\allowbreak 3,\allowbreak 3,\allowbreak 3,\allowbreak 2,\allowbreak 2,\allowbreak 5,\allowbreak 5,\allowbreak 3,\allowbreak 3,\allowbreak 3,\allowbreak 5,\allowbreak 3,\allowbreak 4,\allowbreak 2,\allowbreak 2]$, $(571,169) : [4,\allowbreak 2,\allowbreak 3,\allowbreak 3,\allowbreak 2,\allowbreak 2,\allowbreak 3,\allowbreak 5,\allowbreak 3,\allowbreak 5,\allowbreak 3,\allowbreak 4,\allowbreak 2,\allowbreak 2]$

\end{framed}


We will explain the configuration defined in box \textbf{(9)}. It consists of the 2 $I_8$ fibers, the $I_1$ fiber $F_3$, the sections $E_3, MR_1, E_9$, and the double section $ML_1$. It has $r=21$, $\nu=1$, $t_2=29$, $t_3=1$, and $t_k=0$ for other $k$. We have Chern numbers $\bar c_1^2=21$ and $\bar c_2=13$, and so $\bar c_1^2 / \bar c_2 = 21/13 \approx 1.62$. All curves and intersections are shown in Figure \ref{f5}.

\begin{figure}[htbp]
\centering
\includegraphics[width=10cm]{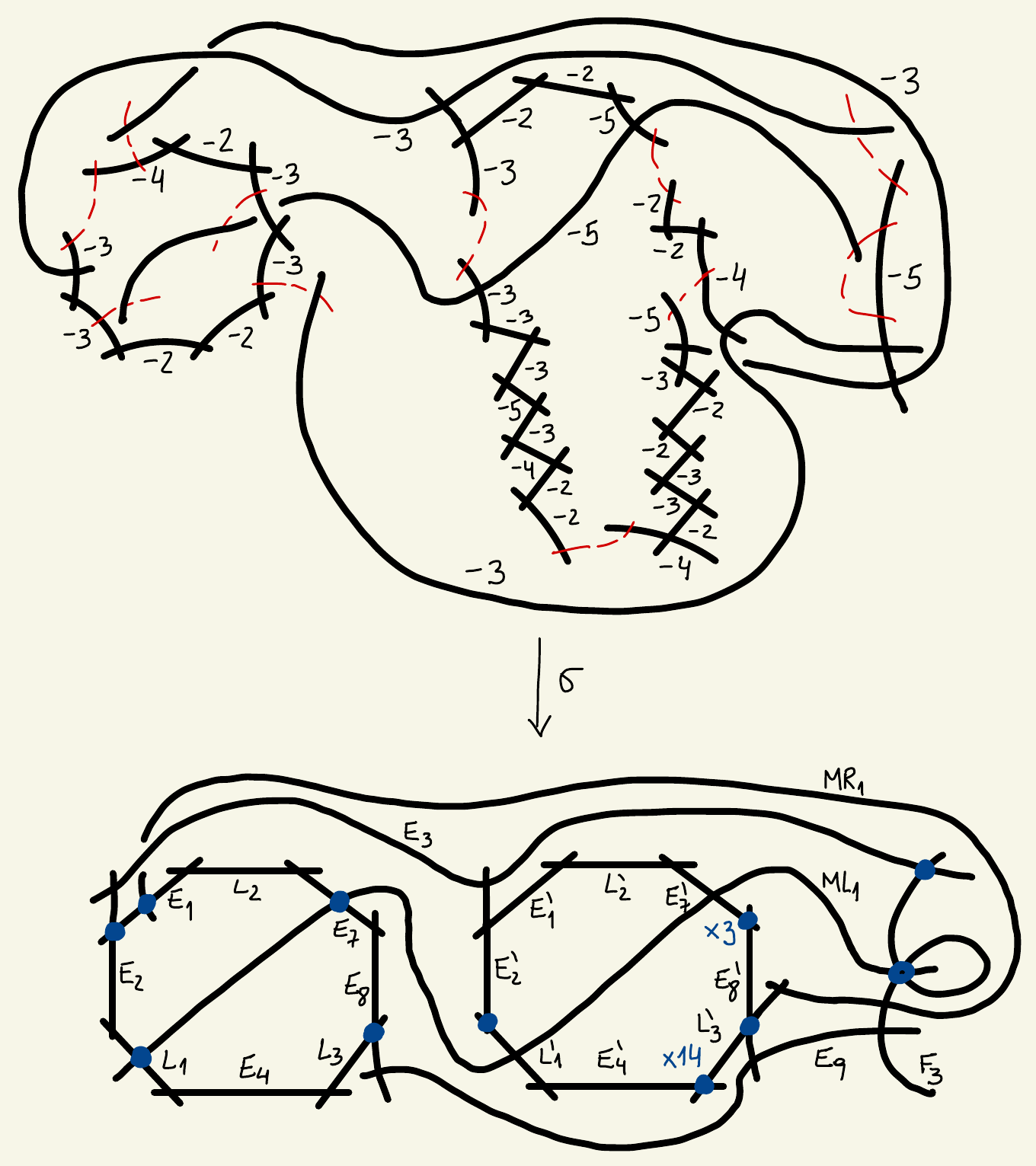}
\caption{The blow-up $\sigma \colon X \to Z$} \label{f5}
\end{figure}

Let $\sigma \colon X \to Z$ be the composition of blow-ups shown in Figure \ref{f5}. The places an numbers of blow-ups are indicated in the picture. We note that in $X$ there are two Wahl chains: $$[4,\allowbreak 2,\allowbreak 3,\allowbreak 3,\allowbreak 2,\allowbreak 2,\allowbreak 3,\allowbreak 3,\allowbreak 3,\allowbreak 3,\allowbreak 2,\allowbreak 2,\allowbreak 5,\allowbreak 5,\allowbreak 3,\allowbreak 3,\allowbreak 3,\allowbreak 5,\allowbreak 3,\allowbreak 4,\allowbreak 2,\allowbreak 2],$$ and $$[4,\allowbreak 2,\allowbreak 3,\allowbreak 3,\allowbreak 2,\allowbreak 2,\allowbreak 3,\allowbreak 5,\allowbreak 3,\allowbreak 5,\allowbreak 3,\allowbreak 4,\allowbreak 2,\allowbreak 2].$$ Let $\pi \colon X \to W$ be the contraction of them.

\begin{theorem}
The rational blowdown of $X$ at these two Wahl chains is an exotic (minimal, symplectic) $3\C \P^2 \# 9 \overline{\C \P^2}$.
\end{theorem}

\begin{proof}
We start computing $K_W^2=-26+22+14=10$. We can easily write $\pi^*(K_W)$ as a positive rational sum of rational curves in a similar way to \textbf{(1)}. Therefore, to show that $K_W$ is nef, we just intersect $\pi^*(K_W)$ with all the relevant $(-1)$-curves, which are the red dashed curves in Figure \ref{f5}, and check that the intersection is positive. This shows that $K_W$ is big and nef, so we can apply Theorem \ref{min} (see Remark \ref{khdskdskd}). To show that it is an exotic $3\C \P^2 \# 9 \overline{\C \P^2}$ via Corollary \ref{exotic}, we need to check that the rational blowdown is simply-connected. For this, we can use the $(-1)$-curve over the intersection of $E'_4$ and $L'_3$ and the fact that the indices of both Wahl singularities have gcd$(19843,571)=1$. This $(-1)$-curve is joining the ends of both Wahl chains, so here we can apply the Seifert-Van Kampen strategy directly.
\end{proof}

Below we have two more examples. Details are left to the reader.

\begin{framed}

\textbf{(10)} $K^2=10$ - $\{L_2,~\allowbreak E_7,~\allowbreak E_8,~\allowbreak L_3,~\allowbreak E_4,~\allowbreak L_1,~\allowbreak E_2,~\allowbreak E_1,~\allowbreak L'_2,~\allowbreak E'_7,~\allowbreak E'_8,~\allowbreak L'_3,~\allowbreak E'_4,~\allowbreak L'_1,~\allowbreak E'_2,~\allowbreak E'_1,~\allowbreak F_3,~\allowbreak E_3,~\allowbreak E_9,~\allowbreak ML_1,~\allowbreak MR_1\}$ - $~F_3 \cap F_3 \cap ML_1$, $~E_7 \cap ML_1$, $~E_4 \cap L_1$, $~E_2 \cap E_3$, $~E_1 \cap MR_1$, $~L'_2 \cap E'_7$, $~E'_8 \cap L'_3$, $~L'_1 \cap ML_1$, $~F_3 \cap E_9$, $[2,1] \times ~E_7 \cap E_8$, $[2,1] \times ~L'_1 \cap E'_2$ - $(513,212) : [3,\allowbreak 2,\allowbreak 4,\allowbreak 3,\allowbreak 3,\allowbreak 3,\allowbreak 5,\allowbreak 3,\allowbreak 3,\allowbreak 3,\allowbreak 2,\allowbreak 4,\allowbreak 2]$, $(121,50) : [3,\allowbreak 2,\allowbreak 4,\allowbreak 3,\allowbreak 5,\allowbreak 3,\allowbreak 3,\allowbreak 2,\allowbreak 4,\allowbreak 2]$

\end{framed}


\begin{framed}

\textbf{(11)} $K^2=10$ - $\{L_2,~\allowbreak E_7,~\allowbreak E_8,~\allowbreak L_3,~\allowbreak E_4,~\allowbreak L_1,~\allowbreak E_2,~\allowbreak E_1,~\allowbreak L'_2,~\allowbreak E'_7,~\allowbreak E'_8,~\allowbreak L'_3,~\allowbreak E'_4,~\allowbreak L'_1,~\allowbreak E'_2,~\allowbreak E'_1,~\allowbreak F_3,~\allowbreak E_3,~\allowbreak E_9,~\allowbreak ML_1,~\allowbreak MR_1\}$ - $~F_3 \cap F_3 \cap ML_1$, $~E_7 \cap E_8$, $~L_1 \cap ML_1$, $~E_1 \cap MR_1$, $~E'_7 \cap ML_1$, $~E'_8 \cap L'_3$, $~L'_1 \cap E'_2$, $~F_3 \cap E_9$, $~E_2 \cap E_1$, $[2,\allowbreak 2,1] \times ~E_2 \cap E_3$, $[2,\allowbreak 2,\allowbreak 3,\allowbreak 4,\allowbreak 2,\allowbreak 2,\allowbreak 1,\allowbreak 4,\allowbreak 2,\allowbreak 3,\allowbreak 5] \times ~LL_1 \cap EE_4$ - $(139,41) : [4,\allowbreak 2,\allowbreak 3,\allowbreak 5,\allowbreak 5,\allowbreak 2,\allowbreak 2,\allowbreak 3,\allowbreak 4,\allowbreak 2,\allowbreak 2]$, $(19309,5695) : [4,\allowbreak 2,\allowbreak 3,\allowbreak 5,\allowbreak 3,\allowbreak 3,\allowbreak 3,\allowbreak 5,\allowbreak 5,\allowbreak 3,\allowbreak 2,\allowbreak 2,\allowbreak 3,\allowbreak 3,\allowbreak 3,\allowbreak 3,\allowbreak 2,\allowbreak 2,\allowbreak 3,\allowbreak 4,\allowbreak 2,\allowbreak 2]$

\end{framed}


\begin{remark}
We point out that examples \textbf{(9)}, \textbf{(10)}, and \textbf{(11)} come from the same configurations of curves! We do not know how to explain this. On the other hand, the canonical class $K_W$ is not ample for the 3 examples. But it can be proved that the proper transforms of $L$ and $P_2$ are the only curves which intersect zero with $K_W$. The contraction of this $A_2$ configuration $L+P_1$ in $W$ gives the canonical model $W'$ of $W$. Therefore $K_{W'}$ is ample, and $W'$ has 2 Wahl singularities and 1 Du Val singularity.
\label{khdskdskd}
\end{remark}

\begin{center}
\bigskip 
\textbf{$K^2=11$}
\bigskip
\end{center}

Here we are able to describe 2 examples from the configuration $D$. The two subconfigurations found in \textbf{(12)} and \textbf{(13)} have no nodal curves, and their singularities are only nodes. Thus, both of them must satisfy $r=24$ and $t_2=35$, as we require $K_W^2=11$ and two Wahl chains (see also \cite[Proposition 3.1]{RU21}). Therefore, for both of them $\bar c_1^2=22$, $\bar c_2=11$, so $\bar c_1^2 / \bar c_2=2$.

One verifies that $K_W$ is ample using the Nakai-Moishezon criterion as before. We can also easily compute $\pi_1(W \setminus \operatorname{Sing}W)=1$ as the indices of the Wahl singularities are coprime, and because there is a $\P_{\C}^1$ connecting the ends of the Wahl chains. In addition, we can also compute the dimension of obstruction. By \cite[Proposition 2.8]{RU21}, the dimension of $H^2(W,T_W)$ is equal to the dimension of the kernel of the intersection matrix of the subconfiguration. It can be checked then, that $h^2(T_W)=4$, and so $h^1(T_W)=0$ by Theorem \ref{obstr} and $W$ is equisingularly rigid.

\begin{framed}

\textbf{(12)} $K^2=11$ - $\{L_2,~\allowbreak E_7,~\allowbreak E_8,~\allowbreak L_3,~\allowbreak E_4,~\allowbreak L_1,~\allowbreak E_2,~\allowbreak E_1,~\allowbreak L'_2,~\allowbreak E'_7,~\allowbreak E'_8,~\allowbreak L'_3,~\allowbreak E'_4,~\allowbreak L'_1,~\allowbreak E'_2,~\allowbreak E'_1,~\allowbreak B_1,~\allowbreak B_2,~\allowbreak L,~\allowbreak E_3,~\allowbreak E_5,~\allowbreak E_6,~\allowbreak MR_1,~\allowbreak ML_2\}$ - $~L_2 \cap E_1$, $~E_7 \cap ML_2$, $~E_4 \cap E_5$, $~E_2 \cap E_3$, $~L'_3 \cap E'_4$, $~L'_1 \cap E'_2$, $~B_1 \cap B_2$, $~B_1 \cap E_5$, $~B_1 \cap E_6$, $~B_2 \cap MR_1$, $~L \cap E_6$, $[2,1] \times ~E'_7 \cap L'_2$, $[3,\allowbreak 3,\allowbreak 2,\allowbreak 2,\allowbreak 3,\allowbreak 2,\allowbreak 3,\allowbreak 3,\allowbreak 3,\allowbreak 3,\allowbreak 2,\allowbreak 3,\allowbreak 2,\allowbreak 1,\allowbreak 3,\allowbreak 4,\allowbreak 3,\allowbreak 3,\allowbreak 3,\allowbreak 4,\allowbreak 5,\allowbreak 3,\allowbreak 3] \times ~L'_2 \cap E_6$ - $(58441,21457) : [3,\allowbreak 4,\allowbreak 3,\allowbreak 3,\allowbreak 3,\allowbreak 4,\allowbreak 5,\allowbreak 3,\allowbreak 3,\allowbreak 2,\allowbreak 6,\allowbreak 3,\allowbreak 3,\allowbreak 2,\allowbreak 2,\allowbreak 3,\allowbreak 2,\allowbreak 3,\allowbreak 3,\allowbreak 3,\allowbreak 3,\allowbreak 2,\allowbreak 3,\allowbreak 2]$, $(42249,15512) : [3,\allowbreak 4,\allowbreak 3,\allowbreak 3,\allowbreak 3,\allowbreak 4,\allowbreak 5,\allowbreak 3,\allowbreak 3,\allowbreak 5,\allowbreak 3,\allowbreak 3,\allowbreak 2,\allowbreak 2,\allowbreak 3,\allowbreak 2,\allowbreak 3,\allowbreak 3,\allowbreak 3,\allowbreak 3,\allowbreak 2,\allowbreak 3,\allowbreak 2]$

\end{framed}


\begin{framed}

\textbf{(13)} $K^2=11$ - $\{L_2,~\allowbreak E_7,~\allowbreak E_8,~\allowbreak L_3,~\allowbreak E_4,~\allowbreak L_1,~\allowbreak E_2,~\allowbreak E_1,~\allowbreak L'_2,~\allowbreak E'_7,~\allowbreak E'_8,~\allowbreak L'_3,~\allowbreak E'_4,~\allowbreak L'_1,~\allowbreak E'_2,~\allowbreak E'_1,~\allowbreak B_1,~\allowbreak E_3,~\allowbreak E_5,~\allowbreak E_9,~\allowbreak M,~\allowbreak MR_1,~\allowbreak MR'_1,~\allowbreak ML_2\}$ - $~L_2 \cap E_7$, $~E_8 \cap L_3$, $~E_4 \cap L_1$, $~E_4 \cap M$, $~E_2 \cap E_1$, $~L'_2 \cap E'_1$, $~E'_8 \cap L'_3$, $~E'_4 \cap E_5$, $~E'_4 \cap M$, $~E'_2 \cap E_3$, $~B_1 \cap E_9$, $[2,1] \times ~ML_2 \cap L'_1$, $[2,\allowbreak 3,\allowbreak 3,\allowbreak 4,\allowbreak 3,\allowbreak 2,\allowbreak 3,\allowbreak 3,\allowbreak 3,\allowbreak 3,\allowbreak 3,\allowbreak 2,\allowbreak 1,\allowbreak 3,\allowbreak 3,\allowbreak 3,\allowbreak 3,\allowbreak 3,\allowbreak 4,\allowbreak 3,\allowbreak 2,\allowbreak 3,\allowbreak 3] \times ~E'_4 \cap L'_3$ - $(88889,33952) : [3,\allowbreak 3,\allowbreak 3,\allowbreak 3,\allowbreak 3,\allowbreak 4,\allowbreak 3,\allowbreak 2,\allowbreak 3,\allowbreak 3,\allowbreak 4,\allowbreak 5,\allowbreak 2,\allowbreak 3,\allowbreak 3,\allowbreak 4,\allowbreak 3,\allowbreak 2,\allowbreak 3,\allowbreak 3,\allowbreak 3,\allowbreak 3,\allowbreak 3,\allowbreak 2]$, $(51584,19703) : [3,\allowbreak 3,\allowbreak 3,\allowbreak 3,\allowbreak 3,\allowbreak 4,\allowbreak 3,\allowbreak 2,\allowbreak 3,\allowbreak 3,\allowbreak 6,\allowbreak 2,\allowbreak 3,\allowbreak 3,\allowbreak 4,\allowbreak 3,\allowbreak 2,\allowbreak 3,\allowbreak 3,\allowbreak 3,\allowbreak 3,\allowbreak 3,\allowbreak 2]$

\end{framed}


\begin{center}
\bigskip 
\textbf{$K^2=12$}
\bigskip
\end{center}

Surprisingly, the next examples for $K^2=12$ (i.e. $b^-=7$) come from nodal configurations of $\P_{\C}^1$'s as well. Moreover, the first two \textbf{(14)}, \textbf{(15)} share the same subconfiguration of $D$, and the same happens with \textbf{(16)}, \textbf{(17)}. As we said before, this phenomenon seems to happen frequently and we have no explanation for it. Since these are nodal configurations of $\P_{\C}^1$'s in a K3 surface, $K_W^2=12$, and there are two Wahl chains, they must have $r=26$, $t_2=38$, $\bar c_1^2=24$, and $\bar c_2=10$, so $\bar c_1^2 / \bar c_2=2.4$ for each of them. One can easily check that $K_W$ is ample in all cases, and that the rational blowdown is indeed simply-connected, as we have coprime indices and the usual Seifert-Van Kampen tricks work out. Moreover, as in the previous $K^2=11$ case, we can compute $h^2(T_W)=6$ and $h^1(T_W)=0$. Hence these $W$ are equisingularly rigid. 

\begin{framed}

\textbf{(14)} $K^2=12$ - $\{L_2,~\allowbreak E_7,~\allowbreak E_8,~\allowbreak L_3,~\allowbreak E_4,~\allowbreak L_1,~\allowbreak E_2,~\allowbreak E_1,~\allowbreak L'_2,~\allowbreak E'_7,~\allowbreak E'_8,~\allowbreak L'_3,~\allowbreak E'_4,~\allowbreak L'_1,~\allowbreak E'_2,~\allowbreak E'_1,~\allowbreak B_2,~\allowbreak L,~\allowbreak C,~\allowbreak E_3,~\allowbreak E_5,~\allowbreak E_6,~\allowbreak E_9,~\allowbreak MR_2,~\allowbreak MR_1,~\allowbreak ML_2\}$ - $~L_2 \cap E_6$, $~E_7 \cap ML_2$, $~L_3 \cap E_4$, $~E_2 \cap E_1$, $~E_1 \cap MR_2$, $~L'_2 \cap E'_1$, $~E'_8 \cap E_9$, $~L'_3 \cap MR_2$, $~L'_3 \cap MR_1$, $~E'_4 \cap E_5$, $~E'_2 \cap E_3$, $~E_4 \cap E_5$, $[2,\allowbreak 2,1] \times ~E_7 \cap E_8$, $[2,\allowbreak 2,1] \times ~E'_2 \cap L'_1$ - $(2687,795) : [4,\allowbreak 2,\allowbreak 3,\allowbreak 3,\allowbreak 2,\allowbreak 3,\allowbreak 4,\allowbreak 3,\allowbreak 5,\allowbreak 3,\allowbreak 2,\allowbreak 3,\allowbreak 4,\allowbreak 3,\allowbreak 4,\allowbreak 2,\allowbreak 2]$, $(436,129) : [4,\allowbreak 2,\allowbreak 3,\allowbreak 3,\allowbreak 2,\allowbreak 3,\allowbreak 5,\allowbreak 3,\allowbreak 4,\allowbreak 3,\allowbreak 4,\allowbreak 2,\allowbreak 2]$

\end{framed}


\begin{framed}

\textbf{(15)} $K^2=12$ - $\{L_2,~\allowbreak E_7,~\allowbreak E_8,~\allowbreak L_3,~\allowbreak E_4,~\allowbreak L_1,~\allowbreak E_2,~\allowbreak E_1,~\allowbreak L'_2,~\allowbreak E'_7,~\allowbreak E'_8,~\allowbreak L'_3,~\allowbreak E'_4,~\allowbreak L'_1,~\allowbreak E'_2,~\allowbreak E'_1,~\allowbreak B_2,~\allowbreak L,~\allowbreak C,~\allowbreak E_3,~\allowbreak E_5,~\allowbreak E_6,~\allowbreak E_9,~\allowbreak MR_2,~\allowbreak MR_1,~\allowbreak ML_2\}$ - $~L_2 \cap E_6$, $~E_7 \cap ML_2$, $~E_8 \cap E_9$, $~L_3 \cap E_4$, $~E_2 \cap E_1$, $~E_1 \cap MR_2$, $~L'_2 \cap E'_1$, $~E'_8 \cap E_9$, $~L'_3 \cap MR_2$, $~L'_3 \cap MR_1$, $~E'_4 \cap E_5$, $~E'_2 \cap E_3$, $[2,\allowbreak 2,1] \times ~E'_2 \cap L'_1$, $[3,\allowbreak 4,\allowbreak 3,\allowbreak 4,\allowbreak 2,\allowbreak 2,\allowbreak 1,\allowbreak 4,\allowbreak 2,\allowbreak 3,\allowbreak 3,\allowbreak 2,\allowbreak 3] \times ~E_7 \cap E_8$ - $(263303,77905) : [4,\allowbreak 2,\allowbreak 3,\allowbreak 3,\allowbreak 2,\allowbreak 3,\allowbreak 3,\allowbreak 2,\allowbreak 3,\allowbreak 3,\allowbreak 2,\allowbreak 3,\allowbreak 4,\allowbreak 3,\allowbreak 5,\allowbreak 3,\allowbreak 2,\allowbreak 3,\allowbreak 4,\allowbreak 3,\allowbreak 4,\allowbreak 3,\allowbreak 4,\allowbreak 3,\allowbreak 4,\allowbreak 2,\allowbreak 2]$, $(436,129) : [4,\allowbreak 2,\allowbreak 3,\allowbreak 3,\allowbreak 2,\allowbreak 3,\allowbreak 5,\allowbreak 3,\allowbreak 4,\allowbreak 3,\allowbreak 4,\allowbreak 2,\allowbreak 2]$

\end{framed}


\begin{framed}

\textbf{(16)} $K^2=12$ - $\{L_2,~\allowbreak E_7,~\allowbreak E_8,~\allowbreak L_3,~\allowbreak E_4,~\allowbreak L_1,~\allowbreak E_2,~\allowbreak E_1,~\allowbreak L'_2,~\allowbreak E'_7,~\allowbreak E'_8,~\allowbreak L'_3,~\allowbreak E'_4,~\allowbreak L'_1,~\allowbreak E'_2,~\allowbreak E'_1,~\allowbreak L,~\allowbreak P_1,~\allowbreak C,~\allowbreak P_2,~\allowbreak E_3,~\allowbreak E_5,~\allowbreak E_6,~\allowbreak MR_1,~\allowbreak MR'_1,~\allowbreak ML_2\}$ - $~L_2 \cap E_7$, $~L_3 \cap E_4$, $~E_2 \cap E_3$, $~E_2 \cap E_1$, $~L'_2 \cap E_6$, $~L'_3 \cap MR_1$, $~E'_4 \cap L'_1$, $~E'_2 \cap E'_1$, $~L \cap P_1$, $~L \cap E_5$, $~P_1 \cap MR'_1$, $~C \cap P_2$, $[2,\allowbreak 2,1] \times ~L'_1 \cap ML_2$, $[3,\allowbreak 3,\allowbreak 3,\allowbreak 4,\allowbreak 2,\allowbreak 2,\allowbreak 1,\allowbreak 4,\allowbreak 2,\allowbreak 3,\allowbreak 3,\allowbreak 3,\allowbreak 3] \times ~MR_1 \cap P_1$ - $(266348,78757) : [4,\allowbreak 2,\allowbreak 3,\allowbreak 3,\allowbreak 3,\allowbreak 3,\allowbreak 2,\allowbreak 2,\allowbreak 3,\allowbreak 3,\allowbreak 3,\allowbreak 3,\allowbreak 2,\allowbreak 3,\allowbreak 5,\allowbreak 3,\allowbreak 4,\allowbreak 3,\allowbreak 3,\allowbreak 3,\allowbreak 5,\allowbreak 3,\allowbreak 3,\allowbreak 3,\allowbreak 4,\allowbreak 2,\allowbreak 2]$, $(487,144) : [4,\allowbreak 2,\allowbreak 3,\allowbreak 3,\allowbreak 3,\allowbreak 3,\allowbreak 5,\allowbreak 3,\allowbreak 3,\allowbreak 3,\allowbreak 4,\allowbreak 2,\allowbreak 2]$

\end{framed}


\begin{framed}

\textbf{(17)} $K^2=12$ - $\{L_2,~\allowbreak E_7,~\allowbreak E_8,~\allowbreak L_3,~\allowbreak E_4,~\allowbreak L_1,~\allowbreak E_2,~\allowbreak E_1,~\allowbreak L'_2,~\allowbreak E'_7,~\allowbreak E'_8,~\allowbreak L'_3,~\allowbreak E'_4,~\allowbreak L'_1,~\allowbreak E'_2,~\allowbreak E'_1,~\allowbreak L,~\allowbreak P_1,~\allowbreak C,~\allowbreak P_2,~\allowbreak E_3,~\allowbreak E_5,~\allowbreak E_6,~\allowbreak MR_1,~\allowbreak MR'_1,~\allowbreak ML_2\}$ - $~E_7 \cap ML_2$, $~L_3 \cap MR'_1$, $~L_3 \cap E_4$, $~E_2 \cap E_3$, $~E_1 \cap MR_1$, $~L'_2 \cap E'_1$, $~L'_3 \cap E'_4$, $~L'_1 \cap E'_2$, $~L \cap E_5$, $~L \cap P_1$, $~P_1 \cap C$, $~L \cap E_6$, $[2,\allowbreak 2,1] \times ~L \cap P_2$, $[2,\allowbreak 2,1] \times ~E_6 \cap L_2$ - $(4,1) : [6,\allowbreak 2,\allowbreak 2]$, $(267721,78962) : [4,\allowbreak 2,\allowbreak 3,\allowbreak 5,\allowbreak 3,\allowbreak 3,\allowbreak 2,\allowbreak 3,\allowbreak 3,\allowbreak 3,\allowbreak 3,\allowbreak 3,\allowbreak 5,\allowbreak 3,\allowbreak 3,\allowbreak 3,\allowbreak 3,\allowbreak 3,\allowbreak 4,\allowbreak 3,\allowbreak 3,\allowbreak 2,\allowbreak 2,\allowbreak 3,\allowbreak 4,\allowbreak 2,\allowbreak 2]$

\end{framed}






\end{document}